\documentclass[onecolumn, a4paper, 12pt, conference]{ieeeconf}      
\IEEEoverridecommandlockouts                              
\let\labelindent\relax

\usepackage[utf8]{inputenc}
\usepackage{graphics} 
\usepackage{amsmath} 
\usepackage{amssymb}  
\usepackage{amsfonts}
\usepackage{comment}
\usepackage{enumitem}
\usepackage{cases,xspace,enumitem}
\usepackage{tikz}
\usetikzlibrary{positioning}
\usepackage{listings}
\usepackage{verbatim}
\usepackage{svg}
\usepackage{cite}

\usepackage{version}
\includeversion{extendedArxiv}

\newcounter{saveenum}

\renewcommand{\theenumi}{(\roman{enumi})}

\usepackage[prependcaption,colorinlistoftodos]{todonotes}

\renewcommand{\theenumi}{(\roman{enumi})}

\newcommand\Item[1][]{%
  \ifx\relax#1\relax  \item \else \item[#1] \fi
  \abovedisplayskip=0pt\abovedisplayshortskip=0pt~\vspace*{-\baselineskip}}

\usepackage{hyperref}
\hypersetup{
    colorlinks=true,
    linkcolor=blue,
    filecolor=magenta,      
    urlcolor=cyan,
    pdftitle={Overleaf Example},
    pdfpagemode=FullScreen,
    }

\usepackage{amsthm}
\newtheorem{thm}{Theorem}
\newtheorem{definition}[thm]{Definition}
\newtheorem{lem}[thm]{Lemma}
\newtheorem{pro}[thm]{Proposition}

\newtheorem{rem}{Remark}

\let\eps\varepsilon

\newcommand{\R}{\mathbb{R}}
\newcommand{\real}{\R}

\newcommand{\bd}{\begin{definition}} 
\newcommand{\ed}{\end{definition}} 
\newcommand{\bp}{\begin{pro}} 
\newcommand{\ep}{\end{pro}} 
\newcommand{\bt}{\begin{thm}} 
\newcommand{\et}{\end{thm}} 
\newcommand{\bi}{\begin{itemize}} 
\newcommand{\ei}{\end{itemize}} 
\newcommand{\bds}{\begin{description}} 
\newcommand{\eds}{\end{description}} 
\newcommand{\beq}{\begin{equation}} 
\newcommand{\eeq}{\end{equation}} 
\newcommand{\blm}{\begin{lem}} 
\newcommand{\elm}{\end{lem}} 
\newcommand{\abs}[1]{\left|#1\right|}

\newcommand{\norm}[1]{\|#1\|}

\newcommand{\norminf}[1]{\|#1\|_\infty}
\newcommand{\derp}[2]{\frac{\partial #1}{\partial #2}}

\newcommand{\trasp}[1]{#1^{\textsf{T}}}

\newcommand{\tr}{\operatorname{tr}}

\newcommand{\diag}[1]{[#1]}

\newcommand{\phimax}{\phi_{\textup{max}}}
\newcommand{\psimax}{\phi_{\textup{max}}}
\newcommand{\umax}{u_{\textup{max}}}
\newcommand{\ubarmax}{\bar u_{\textup{max}}}
\newcommand{\HH}{H}
\newcommand{\hh}{h}
\newcommand{\fr}{\nu}
\newcommand{\hmax}{\hh_{\textup{max}}}
\newcommand{\wmax}{w_{\textup{max}}}

\newcommand{\xmax}{x_{\textup{max}}}
\newcommand{\frmax}{\fr_{\textup{max}}}
\newcommand{\until}[1]{\{1,\dots, #1\}}
\newcommand{\domax}{d_{\textup{max}}}
\newcommand{\bmax}{b_{\textup{max}}}
\newcommand{\amax}{b_{\textup{max}}}
\newcommand{\winv}{\tilde w}
\newcommand{\xinv}{\tilde x}

\newcommand{\gh}{\subscr{g}{h}}
\newcommand{\gfr}{\subscr{g}{f}}
\newcommand{\go}{\subscr{g}{o}}
\newcommand{\gof}{\subscr{g}{of}}
\newcommand{\NR}{c_{\textup{n}}}
\newcommand{\SR}{\subscr{c}{s}}
\newcommand{\OR}{\subscr{c}{o}}
\newcommand{\BO}{\subscr{B}{out}}
\newcommand{\BI}{\subscr{B}{in}}
\newcommand{\BOH}{\trasp{\BO}\Phi(x)}
\newcommand{\BIH}{\trasp{\BI}\Phi(x)}
\newcommand{\BOF}{\trasp{\BO}\Phi(\nu)}
\newcommand{\BIF}{\trasp{\BI}\Phi(\nu)}

\newcommand{\lognorm}[2]{\mu_{#2}\bigl(#1\bigr)}

\newcommand{\subscr}[2]{#1_{\textup{#2}}}

\newcommand{\setdef}[2]{\{#1 \; | \; #2\}}

\newcommand{\map}[3]{#1\colon #2 \rightarrow #3}

\newcommand{\agg}{\operatorname{agg}}
\newcommand{\composite}{\operatorname{cmpst}}
\newcommand{\lrnorm}[2]{\left\|#1\right\|_{#2}}
\newcommand{\Aagg}[1]{|#1|_{\agg}}
\newcommand{\AaggM}[1]{|#1|_{\textup{M}}}
\newcommand{\AM}[1]{|#1|_{\textup{m}}}

\newcommand{\JacM}{\AaggM{J(x,w)}}
\newcommand{\JacMTildeHH}{\tilde J_{\textup{M-HH}}}
\newcommand{\JacMTildeFH}{\tilde J_{\textup{M-FH}}}
\newcommand{\JacMTildeHO}{\tilde J_{\textup{M-HO}}}
\newcommand{\JacMTildeFO}{\tilde J_{\textup{M-FO}}}
\newcommand{\contrateHH}{\subscr{\lambda}{HH}}
\newcommand{\contrateFH}{\subscr{\lambda}{FH}}
\newcommand{\contrateHO}{\subscr{\lambda}{HO}}
\newcommand{\contrateFO}{\subscr{\lambda}{FO}}
\newcommand{\constHH}{\subscr{\tilde c}{HH}}
\newcommand{\constFH}{\subscr{\tilde c}{FH}}
\newcommand{\constHO}{\subscr{\tilde c}{HO}}
\newcommand{\constFO}{\subscr{\tilde c}{FO}}
\newcommand{\1}{\mbox{\fontencoding{U}\fontfamily{bbold}\selectfont1}}

\newcommand{\e}{\mathrm{e}}

\usepackage{titlesec}

\titlespacing*{\section}
{0pt}{3ex plus 1ex minus .2ex}{2ex plus .2ex}
\titlespacing*{\subsection}
{0pt}{3ex plus 1ex minus .2ex}{2ex plus .2ex}
\titlespacing*{\subsubsection}
{0pt}{1.5ex plus 1ex minus .2ex}{1ex plus .2ex}
\titlespacing*{\paragraph}
{0pt}{1.5ex plus 1ex minus .2ex}{1ex plus .2ex}

\title{Modeling and Contractivity of \\
Neural-Synaptic Networks with Hebbian Learning}

\author{Veronica Centorrino$^{1}$, Francesco Bullo$^{2}$, and Giovanni Russo$^{3}$
\thanks{This work was supported in part by AFOSR grant FA9550-22-1-0059. Giovanni Russo wishes to acknowledge financial support by the European Union - Next Generation EU, under PRIN 2022 PNRR, Project “Control of Smart Microbial Communities for Wastewater Treatment”.}
\thanks{$^{1}$Veronica Centorrino is with Scuola Superiore Meridionale, Naples, Italy. {\tt\small veronica.centorrino@unina.it.}}
\thanks{$^{2}$Francesco Bullo is with the Department of Mechanical Engineering and the Center for Control, Dynamical Systems, and Computation,  University of California, Santa Barbara, USA. {\tt\small bullo@ucsb.edu.}}
\thanks{$^{3}$Giovanni Russo is with the Department of Information and Electric Engineering and Applied Mathematics, University of Salerno, Italy. {\tt\small giovarusso@unisa.it.}}
}

\date{}

\begin{document}
\pagestyle{plain}
\maketitle

\begin{abstract}
\normalsize
This paper is concerned with the modeling and analysis of two of the most commonly used recurrent neural network models (i.e., Hopfield neural network and firing-rate neural network) with dynamic recurrent connections undergoing Hebbian learning rules. To capture the synaptic sparsity of neural circuits we propose a low dimensional formulation. We then characterize certain key dynamical properties. First, we give biologically-inspired forward invariance results.
Then, we give sufficient conditions for the non-Euclidean contractivity of the models. Our contraction analysis leads to stability and robustness of time-varying trajectories -- for networks with both excitatory and inhibitory synapses governed by both Hebbian and anti-Hebbian rules. For each model, we propose a contractivity test based upon biologically meaningful quantities, e.g., neural and synaptic decay rate, maximum in-degree, and the maximum synaptic strength. Then, we show that the models satisfy Dale's Principle. Finally, we illustrate the effectiveness of our results via a numerical example.
\end{abstract}

\section{Introduction}
Driven by the massive availability of data in many applications and the increase in computing power, the leading paradigm to train {\em deep} neural networks has become that of feeding them with data, using backpropagation to learn the network weights. This approach has achieved impressive results, spanning from computer vision to natural language processing~\cite{IG-YB-AC:16} and to the end-to-end control of complex video games~\cite{Granturismo_2022}. Nevertheless, despite their biological inspiration and performance achievements, these systems differ from human intelligence in several ways~\cite{BML-TDU-JBT-SJG:16}.
The backpropagation algorithm is not biologically plausible~\cite{IG-YB-AC:16, RCO-YM:00} and this might explain~\cite{BML-TDU-JBT-SJG:16} the poor ability, typical of human intelligence, of certain deep networks to generalize, compose and abstract knowledge from data. In this context, it has been recently shown that models trained via backpropagation can be extremely {\em fragile}~\cite{NG-VA-BA-AH:20}, in the sense that even small changes in the input can produce large changes in the output.

Motivated by these observations, we consider more biologically plausible Recurrent Neural Network (RNN) systems~\cite{MG-OF-PB:12} with dynamic recurrent connections undergoing nonlinear Hebbian Learning (HL) rules~\cite{DOH:49}. Namely, we characterize the dynamic behavior of two of the most commonly used RNN models that we term as Hopfield Neural Network (HNN) and firing-rate neural network (FRNN); see e.g.,~\cite{PD-LFA:05}. For these networks, an important step is that of finding conditions that guarantee their stability and robustness. To do so, we leverage tools from contraction theory~\cite{WL-JJES:98, GR-MDB-EDS:13, AD-SJ-FB:20o}. Indeed, by ensuring contractivity, one also guarantees global exponential convergence and other useful robustness properties.

\subsection{Related Literature}
Over the last few years, there has been a growing interest in the study of biologically plausible learning rules to train neural networks and in finding connections between these rules and backpropagation~\cite{TPL-DC-DBT-CJA:16, BS-YB:17}. For example, HNNs coupled with an HL rule have been shown to be able to learn the underlying geometry of a given set of inputs~\cite{MG-OF-PB:12}. Along these lines, recently, an unsupervised biologically plausible learning rule has been proposed and it has been demonstrated how this rule allows the network to achieve good performance on the MNIST and CIFAR~\cite{DK-JJH:19} datasets; also, in~\cite{CP-DBC:19} it has been shown that neural networks equipped with both Hebbian and anti-Hebbian learning rules can perform a broad range of unsupervised learning tasks.
RNNs naturally emerge when modeling neural processes~\cite{DR-GH-RW:86}. In~\cite{KDM-FF:12} it has been proved, via suitably defined state and input transformations, that the HNN and the FRNN are mathematically equivalent when there is no synaptic dynamics.
While multistability is a key feature of the original Hopfield model~\cite{JJH:84} (with multiple equilibria interpreted as {\em memories}), significant interest has grown over the years in establishing conditions that ensure convergence to a unique equilibrium point~\cite{TC-SIA:01, AD-AVP-FB:21k, YF-TGK:96, MF-AT:95}.
Indeed, when studying RNN models a key problem is that of guaranteeing stability and robustness; see e.g.,~\cite{EN-JC:21a, HZ-ZW-DL:14}. 
A useful tool to study these properties is contraction theory (which precludes multistability). In fact, contracting systems exhibit highly ordered transient and asymptotic behaviors that appear to be convenient in the context of RNNs.
For example: (i) initial conditions are exponentially forgotten~\cite{WL-JJES:98}; (ii) for time-invariant dynamics, there exists a unique globally exponential stable equilibrium~\cite{WL-JJES:98}; (iii) contraction ensures entrainment to periodic inputs~\cite{GR-MDB-EDS:10a} and implies robustness properties such as input-to-state stability, also when there are delayed dynamics~\cite{AD-SJ-FB:20o, SX-GR-RHM:21}.
Moreover, efficient numerical algorithms can be devised for numerical integration and fixed point computation of contracting systems~\cite{SJ-AD-AVP-FB:21f}. Recent reviews of contraction theory are provided in~\cite{HT-SJC-JJES:21, FB:23-CTDS}. Note that contractivity precludes multistability.
An implicit model that uses contraction analysis to allow for a convex parametrization of stable models is presented in~\cite{MR-IM:20}, while in~\cite{SX-GR-RHM:21} contraction-based conditions are given to characterize disturbance rejection properties of HNNs with delays. Contraction theory is also used in~\cite{LK-ME-JJES:22} to find conditions under which assemblies of RNNs are stable.
In the context of networks with adapting synapses undergoing Hebbian rules we recall~\cite{DWD-JJH:92}, where stability of the combined neural and synaptic dynamics is shown via Lyapunov analysis, and~\cite{LK-ML-JJES-EKM:20}, where Euclidean contraction theory is used to study the stability of RNNs with linear coupling between the different nodes and with dynamic synapses undergoing a correlation-based HL rule.
Other works have also shown that contractivity can be effectively studied on Finsler manifolds~\cite{FF-RS:14} or using non-Euclidean norms for large classes of network systems, arising in biological~\cite{ZA-EDS:14, GR-MDB-EDS:10a} and neural~\cite{AD-AVP-FB:21k, HQ-JP-ZBX:01} applications. Finally, we recall~\cite{GR-FW:22}, where contraction is extended to dynamical systems on time scales, evolving on arbitrary, potentially non-uniform, time domains.

\subsection{Contributions} In the context of the above literature, our main contributions are summarized as follows: 
\begin{enumerate}
    \item we study a number of neural-synaptic systems that combine HNN and FRNN (for the neural dynamics) and two different HL rules (for the synaptic dynamics). The first rule, simply termed HL rule, fulfills the biological properties of locality, cooperativity, synaptic depression, and boundedness; the second rule (Oja-like learning rule) fulfills, in addition, a competitiveness~\cite{WG-WK:02} property. To capture the synaptic sparsity of the neural circuits, relying on out-incidence and in-incidence matrices~\cite{FB:22} we propose a low dimensional formulation for the models we analyze, which we term as Hopfield-Hebbian, firing-rate-Hebbian, Hopfield-Oja, and firing-rate-Oja. The models capture networks with both excitatory and inhibitory synapses governed by both Hebbian and anti-Hebbian learning rules.
    \item We give a biologically-inspired forward invariance result of the dynamics of each model and we show that {under suitable conditions} they satisfy Dale’s principle -- an empirical principle~\cite{HD:35} referring to the fact that an individual neuron {has either only excitatory or only inhibitory synapses}. Also, we give sufficient conditions for the contractivity of each coupled model by leveraging non-Euclidean contraction arguments. Remarkably, our sufficient conditions for contractivity and our lower bounds on the contraction rate are both based upon biologically meaningful quantities, i.e., neural and synaptic decay rate, maximum in-degree, and maximum synaptic strength.
    \item Finally, we complement our theoretical results with numerical simulations on a biologically-inspired network~\cite{NM-AMZ-TE:22}. We leverage this network to illustrate the effectiveness of our conditions and use the numerical results as a motivation to outline possible avenues for future research.
\end{enumerate}
We note that our Hopfield-Hebbian model generalizes the ones analyzed in~\cite{DWD-JJH:92, LK-ML-JJES-EKM:20} by relaxing the assumptions of these papers on the sign of the coefficients of the Hebbian rule and on the linearity of the coupling between neurons. Early versions of our results for a special case of the models were presented in~\cite{VC-FB-GR:22g}, where no proofs were given.

\section{Mathematical Background}\label{review}
Let $\R$ and $\R_{\geq 0}$ denote the set of real and non-negative real numbers, respectively.
If $x$, $y \in \R^n$, $x\leq y$ denotes $x_i \leq y_i$ for all $i$. We denote by “$\circ$" the Hadamard product and let $\1_n \in \R^n$ be the $n$-dimensional vector of all ones, $\diag{x} \in \R^{n \times n}$ be the diagonal matrix having $x$ on the main diagonal and $I_n$ be the $n$-dimensional identity matrix. For $A \in \R^{n \times n}$ we denote by $\textup{spec}(A)$, $\det(A)$, $\tr(A)$ its spectrum, determinant and trace, respectively. Also, its \emph{spectral abscissa} is $\alpha(A) : = \max \{\operatorname{Re}(\lambda) | \lambda \in \textup{spec}(A)\}$, where $\operatorname{Re}(\lambda)$ denotes the real part of the complex number $\lambda$. We say that $A$ is \emph{Hurwitz} if $\alpha(A)<0$. We recall that $M \in \R^{n \times n}$ is \emph{Metzler} if $(M)_{ij}>0$, for all $i \neq j$.
Given $A\in \R^{n \times n}$, its \emph{Metzler majorant} $\AM{A}\in \R^{n \times n}$ is defined by
$$\bigl(\AM{A}\bigr)_{ij}=
\begin{cases} a_{ii},  & \text{if }j=i,\\
\abs{a_{ij}}, & \text{if }j\neq i.
\end{cases}
$$
We denote by $\| \cdot \|$ both a norm on $\R^n$ and its induced matrix norm. The \emph{logarithmic norm} (log norm) of $A \in \R^{n \times n}$ induced by $\| \cdot \|$ is (see e.g.,~\cite{CAD-HH:72}) $$\mu(A) := \lim_{h\to 0^{+}} \frac{\norm{I_n + h A} -1}{h}.$$
Finally, whenever it is clear from the context, we omit specifying the dependence of functions on time $t$.

\subsection{Contraction Theory}
Consider the following nonlinear system:
\beq\label{nlsystem}
\dot x = g(x,t), \quad x(0) = x_0,
\eeq
where $x\in\R^n$ is the state of the system, $\map{g}{C\times [0,\infty)}{\R^n}$ is a smooth nonlinear function with $C\subseteq \R^n$ forward invariant set for the dynamics, and $x_0 \in \R^n$ is the initial condition. We assume that $g(x,t)$ is differentiable in $x$, and that $g(x,t)$ as well as its Jacobian with respect to $x$, denoted as $J(x,t) = \partial g(x,t)/\partial x$, are both continuous in $(x,t)$\footnote{For non-open sets, $C$, differentiability in $x$ means that of $g(\cdot,t)$ can be extended as a differential function on some open set that includes $C$ with the continuity hypotheses still holding on this open set~\cite{GR-MDB-EDS:10a}.}.
We assume the existence and uniqueness of a solution to~\eqref{nlsystem}. Next, we give the following definitions
\bd\label{attractive}
A set $\mathcal{A}\subseteq \R^n$ is \emph{attractive} with respect to the system~\eqref{nlsystem} if the trajectories of every solution of~\eqref{nlsystem} starting from any point of $\R^n$ asymptotically converge to the set $\mathcal{A}$.
\ed
\bd \label{def:contracting_system}
Given a norm $\norm{\cdot}$ with associated log norm $\mu$, the system~\eqref{nlsystem} is \emph{strongly infinitesimally contracting} with respect to $\norm{\cdot}$ on a forward invariant, convex, set $C \subset \R^n$, if for some constant $c >0$, referred as \emph{contraction rate}, it holds
\[
\mu\bigl(J(x,t)\bigr) \leq -c,\quad \forall x \in C,\ t \geq 0.
\]
\ed

Given the dynamics~\eqref{nlsystem}, Definition~\ref{def:contracting_system}
implies that for any two trajectories, say $x(\cdot)$ and $y(\cdot)$ rooted from $x_0,y_0\in C$, the following upper bound holds~\cite{GR-MDB-EDS:10a}
$$\|x(t) - y(t)\| \leq \e^{-ct}\|x_0 - y_0\|,\quad \forall t \geq 0,$$
that is, the distance between any two trajectories rooted in $C$ shrinks exponentially with rate $c$.

We refer to~\cite{FB:23-CTDS} for a detailed monograph on contraction theory.

\subsection{Composite Norms}
Consider $r$ positive integers $n_1,\dots,n_r$, such that $\sum_{i=1}^r n_i=n$ and let $\norm{\cdot}_i$ defined on $\real^{n_i}$ be $r$ \emph{local norms} with their
induced log norms $\lognorm{\cdot}{i}$. Also, let $\norm{\cdot}_{\agg}$ on $\real^r$ be an \emph{aggregating norm} with its induced log norm $\lognorm{\cdot}{\agg}$. The \emph{composite norm} $\norm{\cdot}_{\composite}$ is defined as
\[
\norm{x}_{\composite} = 
\lrnorm{
\begin{bmatrix}
x_1 \\
\vdots\\
x_r
\end{bmatrix}
}{\composite} =
\lrnorm{
\begin{bmatrix}
\norm{x_1}_{1}\\
\vdots\\
\norm{x_r}_{r}
\end{bmatrix}
}{\agg}.
\]
In what follows $\lognorm{\cdot}{\composite}$ is the log norm induced by $\norm{\cdot}_{\composite}$.
Next, consider a block matrix $A\in\real^{n\times{n}}$ with blocks $A_{ij}\in\real^{n_i\times{n_j}}$, $i,j \in \until{r}$. The \emph{aggregate majorant} $\Aagg{A}$ and \emph{aggregate Metzler majorant} $\AaggM{A}$ in $\real^{r\times{r}}$ are
\[
\bigl(\Aagg{A}\bigr)_{ij} := \norm{A_{ij}}_{ij}, \quad 
(\AaggM{A})_{ij} :=
\begin{cases}
\lognorm{A_{ii}}{i}, & \text{if }j=i ,\\
\norm{A_{ij}}_{ij},  & \text{if }j\neq i,
\end{cases}
\]
where 
$\norm{A_{ij}}_{ij} = \max\setdef{\norm{A_{ij}y_j}_i}
{y_j\in\real^{n_j}\text{ s.t. }\norm{y_j}_{j}=1}$.
The following result is due to~\cite{JS-CW:62}, see also~\cite{AD-AVP-FB:21k}.
\begin{lem}\label{lem:eta}
Consider a Metzler matrix $M \in \R^{n\times n}$. For any $p \in [1, \infty]$, and $\delta > 0$, define $\eta_{M,p,\delta}\in \R^n_{\geq0}$ by
\[
\eta_{M,p,\delta} = \left(\frac{l_1^{1/p}}{r_1^{1/q}},\dots,\frac{l_n^{1/p}}{r_n^{1/q}} \right),
\]
where $q \in [1, \infty]$ is the conjugate index of $p$, while $l$ and $r \in \R^n_{\geq0}$ are the left and right dominant eigenvectors of $M+\delta \1_n \trasp{\1_n}$.
Then for each $\eps>0$ there exists $\delta > 0$ such that
\begin{enumerate}
\item $\alpha(M) \leq \mu_{\diag{\eta_{M,p,\delta}}}\bigl(M\bigr) \leq \alpha(M) + \eps$,
\item if $M$ is irreducible, then $\alpha(M) = \mu_{\diag{\eta_{M,p,0}}}\bigl(M\bigr)$.
\end{enumerate}
\end{lem}
Finally, we report results that can be found, under different technical statements, in~\cite{TS:75, SX-GR-RHM:21} with the upper bound for the log norm in~\ref{majorant:Strom-bounds} introduced in~\cite{GR-MDB-EDS:13}.
\begin{thm}\label{thm:composite-norm-majorants}
For any set of local norms $\norm{\cdot}_i$, $i\in\until{r}$, consider a monotonic aggregating norm $\norm{\cdot}_{\agg}$ over a decomposition of $\R^{n}$ and a matrix $A\in\real^{n\times{n}}$. Then:
\begin{enumerate}
\Item\label{majorant:Strom-bounds}
\begin{align*}
\max_{i\in\until{r}}\norm{A_{ii}}_i & \leq \norm{A}_{\composite} \leq \norm{\Aagg{A}}_{\agg},\\
\max_{i\in\until{r}}\lognorm{A_{ii}}{i} & \leq \lognorm{A}{\composite} \leq \lognorm{\AaggM{A}}{\agg};
\end{align*}
\item\label{majorant:Hurwitz}
if $\AaggM{A}$ is Hurwitz, then $A$ is Hurwitz.
\end{enumerate}
\end{thm}
\subsection{Out-incidence and In-incidence Matrices}
Let $G$ be a weighted directed graph with $n$ nodes and $m$ edges, and let $V = \{1,\dots,n\}$ and $E = \{1,\dots,m\}$ be the set of nodes and edges of $G$, respectively. We write $e = (i,j)$, $i, j \in V$, when we want to emphasize the nodes associated with the edge $e$, and we refer to $i$ as the \emph{tail} and to $j$ as the \emph{head} of $e$. According to the context, with a little abuse of notation, we let $e$ denote both an ordered pair $(i,j)$ as well as an element of $E$.
We let $\{a_e\}_{e \in E}$ be the set of weights for the edges of $G$.
The \emph{topological} \emph{in-degree} and \emph{out-degree} of a vertex $i \in V$, are the number of in-neighbors and out-neighbors of $i$, respectively. The \emph{maximum topological in-degree} and \emph{maximum topological out-degree} of the graph $G$, are the highest topological in-degree and out-degree among all vertices in $G$, respectively.
The \emph{adjacency matrix} $A \in\R^{n\times n}$ is defined as follows: for each edge $e = (i,j) \in E$, the entry $(i,j)$ of $A$ is equal to the weight $a_e$ of the edge $(i,j)$, and all other entries of $A$ are equal to zero. The weight matrix $\mathcal{A}\in\R^{m\times m}$ is the diagonal matrix of edge weights, that is $\mathcal{A} := \diag{\{a_e\}_{e \in E}}$.
Finally, for any node $i\in V$ and edge $e \in E$, the \emph{out-incidence matrix} $\BO \in \{0, 1\}^{n\times m}$ and \emph{in-incidence matrix} $\BI \in \{0, 1\}^{n\times m}$ are respectively defined by
\begin{align}
\bigl(\BO \bigr)_{ie} &=
\begin{cases}
1 & \text{ if node $i$ is the head of edge $e$}, \\
0 & \text{ otherwise}, 
\end{cases}
\label{eq:out_incidence_matrix}\\
\bigl(\BI \bigr)_{ie} &=
\begin{cases}
1 & \text{ if node $i$ is the tail of edge $e$},\\
0 & \text{ otherwise}.
\end{cases}
\label{eq:in_incidence_matrix}
\end{align}
Note that, by construction, the matrices $\BO$ and $\BI$ have unit row sums, thus $\trasp{\BO}$ and $\trasp{\BI}$ have unit column sums.
The following result can be found in~\cite[Ex. 9.4]{FB:22}.
\begin{pro}\label{prop:BI_BO}
Consider the matrices $A\in\R^{n\times n}$, $\mathcal{A}\in\R^{m\times m}$, $\BO$, and $\BI \in \{0,1\}^{n\times m}$. Then:
\begin{enumerate}
\item for each $x\in \R^n$ and $e \in E$ of the form $e = (i, j)$,
\beq \label{eq:i_j_component_of}
\bigl(\trasp{\BO}x\bigr)_{e} = x_j, \text{ and } \ \bigl(\trasp{\BI}x\bigr)_{e} = x_i.
\eeq
\item the following identity holds
\begin{align}
A &= \BI \mathcal{A} \trasp{\BO}
\label{eq:A = BO A_m BI^T}.
\end{align}
\item $\norminf{\trasp{\BO}} = \norminf{\trasp{\BI}} = 1$, $\norminf{\BO}$ and $\norminf{\BI}$ are the maximum topological out-degree and maximum topological in-degree of $G$, respectively.
\end{enumerate}
\end{pro}
In what follows, we let $\domax := \norminf{\BI}$ and we refer to Section~\ref{Numerical Example} for an example with a visualization of these graph theoretic concepts.

\section{Modeling}\label{modeling}
We first introduce the dynamical rules governing the neural masses and synaptic weights. Then, we present the coupled neural-synaptic dynamical systems we analyze. For each neuron, say neuron $i$, we denote its mean membrane potential at time $t$ by $x_i(t)\in \R$. The instantaneous firing-rate for the $i$-th neuron, say $\nu_i(t)\in \R_{\geq 0}$, is linked to the membrane potential through a nonlinear non-negative monotonically increasing function $\phi: \R \to \R_{\geq 0}$. That is, $\nu_i = \phi(x_i)$ and we term the function $\phi$ as activation function in what follows.
A popular modeling choice~\cite{WG-WMK-RN-LP:14} is to pick $\phi(x_i)$ as a sigmoid.
\subsection{Neural Dynamics}
\subsubsection{Hopfield Neural Network}
For each neural mass $i$ we consider the continuous-time HNN of the form:
\beq \label{eq:hopfield_dynamic}
\dot x_i = - \NR x_i + \sum_{j=1}^n W_{ij}\phi(x_j)+u_i.
\eeq
The first term on the right-hand side of~\eqref{eq:hopfield_dynamic} models the intrinsic dynamics of neuron $i$ and $\NR$ is its decay rate. The second term models the coupling of neuron $i$ with the other neurons; namely, $W_{ij}: \R_{\geq 0} \to \R$ denotes the effective time-dependent synaptic weight of the signal transmitted from a pre-synaptic neuron $j$ to a post-synaptic neuron $i$, and $\phi: \R \to \R_{\geq 0}$ is the activation function. Finally, $u_i: \R_{\geq 0} \to \R$ is a time-dependent external stimulus to neuron $i$.
\subsubsection{Firing-rate Neural Network}
For each neural mass $i$ we consider the continuous-time FRNN of the form:
\beq \label{eq:firing_rate_dynamic}
\dot \fr_i = - \NR \fr_i + \phi\Biggl(\sum_{j=1}^n W_{ij}\fr_j+u_i\Biggr),
\eeq
where, as in~\eqref{eq:hopfield_dynamic}, the decay rate is $\NR$, the synaptic weight is $W_{ij}: \R_{\geq 0} \to \R$, the external stimulus is $u_i: \R_{\geq 0} \to \R$, and the activation function is $\phi: \R \to \R_{\geq 0}$.
\begin{rem}
The activation functions in~\eqref{eq:hopfield_dynamic} and~\eqref{eq:firing_rate_dynamic} can be different, but they have the same properties.
Therefore, to streamline our derivation, we are using the same symbol for both activation functions.
\end{rem}
\begin{rem}
We term the \textup{RNN}~\eqref{eq:firing_rate_dynamic} as {firing-rate neural network} because when the activation function is non-negative, the positive orthant is forward-invariant and the state $\nu_i$ is interpreted as a firing rate. In contrast, in~\eqref{eq:hopfield_dynamic}, the state $x_i$ can be either positive or negative, and thus $x_i$ is interpreted as a membrane potential. Note that~\eqref{eq:hopfield_dynamic} has the same form as the original Hopfield model~\cite{JJH:84} with the key difference that the synaptic matrix is not assumed symmetric. Despite the absence of the symmetry assumption, with a slight abuse of terminology, we also term the Hopfield-like neural network~\eqref{eq:hopfield_dynamic} as Hopfield neural network. This terminology is consistent with the terminology used in, e.g.,~\cite{TC-SIA:01, YF-TGK:96, MG-OF-PB:12}.
\end{rem}
\subsection{Synaptic Dynamics}
We consider the case where the synaptic weights evolve according to continuous time Hebbian learning rules. Following Hebb’s postulate~\cite{DOH:49}, the weight between two neurons should increase if both neurons are simultaneously active; our modeling, capturing this aspect, is based upon the framework presented in~\cite{WG-WK:02}, where a number of formulations of Hebbian learning are reviewed. The first synaptic rule we consider is modeled via a dynamic of the form:
\beq \label{eq:hebbian_dynamic}
\dot W_{ij} = \HH_{ij}\phi(y_i)\phi(y_j) - \SR W_{ij} + \bar U_{ij},
\eeq
where $y_i$ can be either the membrane potential or the firing-rate (in what follows we use $y_i$ when we want to refer to both state variables). As described next, the above model satisfies the properties of locality, cooperativity, and synaptic depression, which are biologically-inspired requirements for any model that aims to capture Hebbian synaptic plasticity,~\cite{WG-WK:02}. Indeed, the first term on the right-hand side of~\eqref{eq:hebbian_dynamic} describes the cooperation between pre- and post-synaptic activity: in the absence of external stimuli both the pre- and post-synaptic neurons must be active to induce a weight increase or decrease (\emph{cooperativity property}). The coefficient $\HH_{ij}\in \R$ is defined so that a non-zero entry corresponds to an existing synaptic connection and a corresponding evolution of the synaptic weight. Specifically, $H_{ij}$ describes the topology of the network (and this is constant over time), while $W_{ij}(t)$ describes the time-varying evolution of the corresponding synaptic weight.
The second term is a decay factor ($\SR>0$) that prevents the weights from diverging (\emph{synaptic depression property}). Finally, $\bar U_{ij}: \R_{\geq 0} \to \R$ is a time-dependent external stimulus, e.g., it can represent some exogenous phenomena. Moreover, the rule modeled in~\eqref{eq:hebbian_dynamic} is also local in the sense that changes in $W_{ij}$ only depend on the activities of neurons $j$ and $i$ (\emph{locality property}). For our derivations, it is useful to define:
\beq\label{def:hmax}
\hmax :=\displaystyle \max_{i,j \in \until{n}}\abs{\HH_{ij}}.
\eeq 
Moreover, following~\cite{WG-WK:02}, we give the following:
\bd
We call an HL rule with $H_{ij}>0$ \emph{Hebbian learning}, and a rule with $H_{ij}<0$ \emph{anti-Hebbian learning}.
\ed
The second model for Hebbian learning we consider also fulfills a {\em competivity} property. This is a further useful feature of learning
rules implying that, if some weights grow, they do so at the expense of others. To capture this feature, we leverage the following Oja's like learning rule~\cite{EO:82}:
\beq \label{eq:oja_dynamic}
\dot W_{ij} = \HH_{ij}\phi(y_i)\phi(y_j) - \bigl(\SR + \OR \phi^2(y_i)\bigr) W_{ij} + \bar U_{ij},
\eeq
with $\OR>0$.
We observe that if $\OR \! = 0$~\eqref{eq:oja_dynamic} reduces to~\eqref{eq:hebbian_dynamic}.
\begin{rem}
When $y_i$ is the membrane potential and there are no external stimuli, equations~\eqref{eq:hebbian_dynamic} and~\eqref{eq:oja_dynamic} become the ones found in~\cite{WG-WK:02}. When $y_i$ is the firing-rate, through the activation function we are introducing a non-linearity.
We emphasize that to streamline our derivations we are using the same notation for the activation functions across the models, since they verify the same properties. However the activation function in~\eqref{eq:hopfield_dynamic},~\eqref{eq:firing_rate_dynamic},~\eqref{eq:hebbian_dynamic}, and~\eqref{eq:oja_dynamic} can be different.
\end{rem}

\subsection{Coupled Neural Synaptic Models}
\label{sec:coupled_neural_synaptic_model}
Consider an RNN of $n$ neurons with dynamic synapses and fixed topology of interactions. That is, the coefficients of $H = (H_{ij})_{i,j} \in \R^{n\times n}$ describing the Hebbian and anti-Hebbian learning connections are constant.
We make no assumptions on the relative timescales of synaptic and neural activity. We now present the models that will be the subject of our study. These models are obtained by combining the above neural and synaptic dynamics.
\subsubsection{Hopfield-Hebbian Model}
The coupled \emph{Hopfield-Hebbian model} is obtained by combining the HNN~\eqref{eq:hopfield_dynamic}, and the HL rule~\eqref{eq:hebbian_dynamic}, with initial neural and synaptic conditions $x_i(0)\in \R$ and $W_{ij}(0)\in \R$, respectively. 
For our analysis, it is useful to write this system in vector form:
\beq \label{eq:complete_dynamic_system}
\left \{
\begin{aligned}
\displaystyle \dot x &= - \NR x + W\Phi(x) + u, \\
\displaystyle \dot W &= \HH \circ \Phi(x)\trasp{\Phi(x)} - \SR W + \bar U,
\end{aligned}
\right.
\eeq
with initial neural and synaptic conditions $x(0) := x_0\in \R^n$ and $W(0) := W_0 \in \R^{n\times n}$, respectively. In~\eqref{eq:complete_dynamic_system} $x, u  \in \R^n$ are the state and the external neural stimuli, respectively, $\Phi: \R^n \to \R^n_{\geq 0}$ is the term by term application of the function $\phi$, i.e., $\Phi(x)_i = \phi(x_i)$, $\bar U \in \R^{n\times n}$ are the external synaptic stimuli.
\subsubsection{Firing-rate-Hebbian Model}
The coupled \emph{firing-rate-Hebbian model} is obtained by combining the FRNN~\eqref{eq:firing_rate_dynamic}, and the HL rule~\eqref{eq:hebbian_dynamic}, with initial neural and synaptic conditions $\fr_i(0)\in \R$ and $W_{ij}(0)\in \R$, respectively.
In vector form, the system is:
\beq \label{eq:firing_rate_complete_dynamic_system}
\left \{
\begin{aligned}
\displaystyle \dot \fr &= - \NR \fr + \Phi\bigl(W\fr + u\bigr), \\
\displaystyle \dot W &= \HH \circ \Phi(\fr)\trasp{\Phi(\fr)} - \SR W + \bar U,
\end{aligned}
\right.
\eeq
with initial neural and synaptic conditions $\fr(0) := \fr_0 \in \R^n$, and $W_0$, respectively. In~\eqref{eq:firing_rate_complete_dynamic_system} $\fr \in\R^n$ is the vector of the firing-rates, while for notational convenience the other terms are defined consistently with~\eqref{eq:complete_dynamic_system}.
\subsubsection{Hopfield-Oja Model}
The coupled \emph{Hopfield-Oja model} is obtained by combining the HNN~\eqref{eq:hopfield_dynamic}, and the Oja's like synaptic plasticity rule~\eqref{eq:oja_dynamic}, with initial neural and synaptic conditions $x_i(0)$ and $W_{ij}(0)$, respectively. Using the same notation as in~\eqref{eq:complete_dynamic_system} we write this system in vector form as
\beq \label{eq:complete_dynamic_system_oja}
\left \{
\begin{aligned}
\displaystyle \dot x &= - \NR x + W\Phi(x) + u, \\
\displaystyle \dot W &= \HH \circ \Phi(x)\trasp{\Phi(x)} - \bigl(\SR I_n + \OR \diag{\Phi(x)}\diag{\Phi(x)}\bigr)W+\bar U,\\
\end{aligned}
\right.
\eeq
with initial neural and synaptic conditions $x_0$ and $W_0$, respectively.
\begin{extendedArxiv}
\subsubsection{Firing-rate-Oja Model}
The coupled \emph{firing-rate-Oja model} is obtained by combining the FRNN~\eqref{eq:firing_rate_dynamic}, and the Oja's like synaptic plasticity model~\eqref{eq:oja_dynamic}, with initial neural and synaptic conditions $\nu_i(0)$ and $W_{ij}(0)$, respectively. Using the same notation as in~\eqref{eq:firing_rate_complete_dynamic_system} we write this system in vector form as
\beq \label{eq:complete_dynamic_system_fr_oja}
\left \{
\begin{aligned}
\displaystyle \dot \fr &= - \NR \fr + \Phi\bigl(W\fr + u\bigr), \\
\displaystyle \dot W &= \HH \circ \Phi(\nu)\trasp{\Phi(\nu)} - \bigl(\SR I_n + \OR \diag{\Phi(\nu)}\diag{\Phi(\nu)}\bigr)W+\bar U,
\end{aligned}
\right.
\eeq
with initial neural and synaptic conditions $\nu_0$ and $W_0$, respectively.
\end{extendedArxiv}
\subsubsection*{\textbf{Assumptions}}
For every neuron $i$ we assume that the activation function satisfies
\begin{enumerate}[label=\textup{($A$\arabic*)}, leftmargin=1.4 cm,noitemsep]
\item\label{ass:phimax} $\displaystyle 0 \leq \phi(y_i) \leq \phimax :=\max_{i\in\until{n}}\sup_{t}\phi(y_i(t))$,
\item\label{ass:phiprime} $0\leq\phi'(y_i)\leq1$.
\setcounter{saveenum}{\value{enumi}}
\end{enumerate}
Moreover, for every $i$ and $j$, we assume that the external stimuli are such that
\begin{enumerate}[label=\textup{($A$\arabic*)}, leftmargin=1.4 cm,noitemsep]
\setcounter{enumi}{\value{saveenum}}
\item\label{ass:umax} $\displaystyle \abs{u_i(t)} \leq \umax:=\max_{i\in\{1,\dots,n\}}\sup_{t}u_i(t)$,
\item\label{ass:ubarmax} $\displaystyle \abs{\bar U_{ij}(t)} \leq \ubarmax:=\max_{i,j\in\{1,\dots,n\}}\sup_{t}\bar U_{ij}(t)$.
\end{enumerate}
\begin{rem}
The assumptions on bounded activation function and external stimuli are used to prove the forward invariance results for the coupled dynamics. While all the assumptions are used for the contraction analysis. As we will show, for the results on the firing-rate-Hebbian and firing-rate-Oja models, assumption~\ref{ass:umax} is not needed. We also remark that the assumptions are not restrictive in practice. Indeed, widely used activation functions (e.g., sigmoid) satisfy, possibly after rescaling, assumptions~\ref{ass:phimax} and~\ref{ass:phiprime}. It is also physically plausible that the external stimuli are bounded.
\end{rem}
\section{Dynamical Properties of the Models}\label{Dynamical properties of the model}
The models of Section~\ref{modeling} have $n\times n^2$ variables -- $n$ neurons and $n^2$ synaptic connections. However, synaptic connectivity in the brain is sparse compared to the number of neurons~\cite{PH-LC-XG-RM-CJH-VJW-OS:08, LK-ML-JJES-EKM:20}.
To exploit this sparsity, we propose low-dimensional reformulations of the above models. These reformulations, which leverage the out-incidence and in-incidence matrices from Section~\ref{review}, are then used to give biologically-inspired forward invariance results and to obtain sufficient conditions for non-Euclidean contractivity of the models. Finally, we show that the models fulfill Dale's principle under suitable conditions.
\subsection{Low Dimensional Reformulations}\label{sec:reformulations}
Let $m$ be the number of synaptic connections in~\eqref{eq:complete_dynamic_system}. To obtain the reduced formulation, we pick the $m$ nonzero elements of $H$, say $H_{ij}$, and the corresponding elements of $W$ and $\bar U$, say $W_{ij}$ and $\bar{U}_{ij}$. We then vectorize these elements in $\hh\in \R^m$, $w\in \R^m$ and $\bar u \in \R^m$, respectively.

We stress that, in our notation, $W_{ij}$ is the synaptic weight of the signal transmitted from a pre-synaptic
neuron $j$ to a post-synaptic neuron $i$. These connections define a graph, which has a $n\times n$ adjacency matrix having as element $(i,j)$ the weight $W_{ij}$. Therefore applying~\eqref{eq:A = BO A_m BI^T} we have $ W  = \BI \diag{w} \trasp{\BO}$, where $\BO$ and $\BI$ are defined as in~\eqref{eq:out_incidence_matrix} and~\eqref{eq:in_incidence_matrix}. Moreover, from~\eqref{eq:i_j_component_of} for each edge (i.e., synaptic connection) of the form $e=(i,j)$, we get $\bigl(\trasp{\BO} \Phi(y)\bigr)_{e} = \phi(y_j)$, and $\bigl(\trasp{\BI} \Phi(y)\bigr)_{e} = \phi(y_i)$.

Substituting the above identities in the full dimensional coupled neural-synaptic models introduced in Section~\ref{sec:coupled_neural_synaptic_model} we obtain the corresponding low dimensional reformulations.
It is worth remarking that in each case we obtain a system with $n \times m$ variables, with $m\ll n^2$, instead of a system with $n \times n^2$ variables. Specifically, we obtain:
\subsubsection{Hopfield-Hebbian Model Reformulation}
\beq \label{eq:complete_dynamic_system_nodes_edges}
\left \{
\begin{aligned}
\displaystyle \dot x &= - \NR x + \BI \diag{w} \trasp{\BO} \Phi(x) + u, \\
\displaystyle \dot w &= \hh \circ \BOH \circ \BIH - \SR w + \bar u,
\end{aligned}
\right.
\eeq
with $x(0) := x_0\in \R^n$ and $w(0) := w_0 \in \R^{m}$. The components of $w_0$ are the $m$ elements of $W_0$ having non zero $H_{ij}$'s.
\subsubsection{Firing-rate-Hebbian Model Reformulation}
\beq \label{eq:firing_rate_complete_dynamic_system_nodes_edges}
\left \{
\begin{aligned}
\displaystyle \dot \fr &= - \NR \fr + \Phi\bigl(\BI \diag{w} \trasp{\BO} \fr + u\bigr),\\
\displaystyle \dot w &= \hh \circ \BOF \circ \BIF - \SR w  + \bar u,
\end{aligned}
\right.
\eeq
with $\fr(0) := \fr_0\in \R^n$ and $w_0 \in \R^{m}$ defined consistently with the initial conditions in \eqref{eq:complete_dynamic_system_nodes_edges}.
\subsubsection{Hopfield-Oja Model Reformulation}
\beq \label{eq:complete_dynamic_system_nodes_edges_oja}
\left \{
\begin{aligned}
\displaystyle \dot x =& - \NR x + \BI \diag{w} \trasp{\BO} \Phi(x) + u, \\
\displaystyle \dot w =& \hh \circ \BOH \circ \BIH - \bigl(\SR I_m + \OR \diag{\BIH}\diag{\BIH} \bigr)w+\bar u,
\end{aligned}
\right.
\eeq
with $x_0\in \R^n$ and $w_0 \in \R^{m}$ defined consistently with the initial conditions in \eqref{eq:complete_dynamic_system_nodes_edges}.
\begin{extendedArxiv}
\subsubsection{Firing-rate-Oja Model Reformulation}
\beq \label{eq:firing_rate_complete_dynamic_system_nodes_edges_oja}
\left \{
\begin{aligned}
\displaystyle \dot \fr &= - \NR \fr + \Phi\bigl(\BI \diag{w} \trasp{\BO} \fr + u\bigr),\\
\displaystyle \dot w &= \hh \circ \BOH \circ \BIH - \bigl(\SR I_m + \OR \diag{\BIH}\diag{\BIH} \bigr)w+\bar u,
\end{aligned}
\right.
\eeq
with $\fr_0\in \R^n$ and $w_0 \in \R^{m}$ defined consistently with the initial conditions in~\eqref{eq:firing_rate_complete_dynamic_system_nodes_edges}.
\end{extendedArxiv}

\subsection{Proving Bounded Evolutions}\label{bouded evolutions}
All biological neurons eventually saturate for high input values and the synaptic weights remain bounded.
Inspired by these properties, we now investigate whether the solutions of the models of Section \ref{sec:reformulations} are bounded. To state our results we define the following sets:
\begin{align*}
\mathcal{X} &:= \{x \in\R^{n} \;|\; |x_{i}| \leq \xmax,\ i\in\until{n} \},\\
\mathcal{V} &:= \{\fr \in\R^{n} \;|\; |\fr_{i}| \leq \frmax,\ i\in\until{n} \},\\
\mathcal{W} &:= \{W \in\R^{n\times{n}} \;|\; |W_{ij}| \leq \wmax,\ i,j \in\until{n} \},
\end{align*}
where $\wmax {:=} (\hmax \phimax^2 {+} \ubarmax)/\SR$ is the maximum synaptic weight value, $\xmax{:= }(\umax+ \domax \phimax \wmax)/\NR$ is the maximum membrane potential, and $\frmax {:=} \psimax /\NR$ is the maximum firing-rate.
With the next result we show that the solutions of the Hopfield-Hebbian model have bounded evolutions when~\ref{ass:phimax},~\ref{ass:umax},~\ref{ass:ubarmax} hold.

\begin{lem}[Bounded evolutions Hopfield-Hebbian] \label{lem:bonded_ev_HH}
Consider model~\eqref{eq:complete_dynamic_system_nodes_edges} and let~\ref{ass:phimax},~\ref{ass:umax},~\ref{ass:ubarmax} hold.
Then, the set $\mathcal{X}\times\mathcal{W}$ is forward invariant and attractive, in the sense that, for every neuron $i\in \until{n}$, and every edge $e\in \until{m}$, the following inequalities hold
\begin{align}
\abs{x_i(t)} &\leq \bigl(\abs{x_i(0)} - \xmax\bigr)\e^{-\NR t} + \xmax, &t \geq 0,\label{in:bounded_evolution_x}\\
\abs{w_e(t)} &\leq \bigl(\abs{w_e(0)} - \wmax\bigr)\e^{-\SR t} + \wmax, &t \geq 0.
\label{in:bounded_evolution_w}
\end{align}
\end{lem}
\begin{proof}
Let $\trasp{(\trasp{\xinv}, \trasp{\winv})}$ be a solution of~\eqref{eq:complete_dynamic_system_nodes_edges} having initial conditions $\trasp{(\trasp{\xinv_0}, \trasp{\winv_0})}$ in the set $\mathcal{X}\times\mathcal{W}$. Considering the synaptic dynamics in~\eqref{eq:complete_dynamic_system_nodes_edges} written in component, for each edge $e$ we have $\dot{\winv}_e(t) = h_{e} \bigl(\BOH \bigr)_e \bigl(\BIH\bigr)_e - \SR \winv_e(t) + \bar u_e$, for all $t \geq 0$.
Since~\ref{ass:phimax} and~\ref{ass:ubarmax} hold and $\trasp{\BO}$ and $\trasp{\BI}$ are unit column sum matrices, we get the upper bound
\beq \label{proof_boundness:din_w}
\dot{\winv}_e(t) \le \hmax\phimax^2 - c_s \winv_e(t) + \ubarmax,
\eeq 
Next, let $v(t) := \abs{\winv_e(t)}$, for all $t \geq 0$, with $v(0) = \abs{\winv_e(0)}$.
From~\eqref{proof_boundness:din_w} we get
\begin{align*}
{D^{+}\abs{\winv_e(t)}} &= \limsup_{k \to 0^{+}} \frac{1}{k}\bigl(\abs{\winv_e(t+k)}-\abs{\winv_e(t)}\bigr)\\
&\leq \limsup_{k \to 0^{+}} \frac{\abs{\winv_e(t) + k\bigl(\hmax\phimax^2 -\SR\winv_e(t) + \ubarmax\bigr)} -\abs{\winv_e(t)}}{k}\\
&\leq \limsup_{k \to 0^{+}} \frac{\norm{I_m - k\SR I_m}-1}{k}\abs{\winv_e(t)} + \hmax\phimax^2 + \ubarmax\\
&= \mu(-\SR I_m) \abs{\winv_e(t)} + \hmax\phimax^2 + \ubarmax\\
&= -\SR\abs{\winv_e(t)} + \hmax\phimax^2 + \ubarmax.
\end{align*}
Therefore $D^{+}\abs{\winv_e(t)} \leq -\SR\abs{\winv_e(t)} + \hmax\phimax^2 + \ubarmax$, for all $t \geq 0$. Consider the function $u: \R_{\geq 0} \to \R$ and define the differential equation
\[
\dot u(t) = g(u,t) := -\SR u(t) + \hmax\phimax^2+\ubarmax, \quad u(0) = \abs{\winv_e(0)}.
\]
Its solution is $u(t) = (\abs{\winv_e(0)} - \wmax)\e^{-\SR t} + \wmax$, where $\wmax := \left(\hmax\phimax^2 + \ubarmax\right)/\SR$.
Applying the comparison Lemma~\cite[pp.~102-103]{HKK:02} we have $v(t) \leq u(t)$, for all $t \geq 0$, i.e.,
\beq \label{eq:w_comparison_lemma}
\abs{\winv_e(t)} \leq \bigl(\abs{\winv_e(0)} - \wmax \bigr)\e^{-\SR t} + \wmax.
\eeq
Being $\winv_e(0) \in \mathcal{W}$ we get $\abs{\winv_e(t)} \leq \wmax$, hence $\winv_e(t) \in \mathcal{W}$, for all $t \ge 0$ and edges $e$.
Moreover, considering the neural dynamics in~\eqref{eq:complete_dynamic_system_nodes_edges} written in component, for each $i$ we have $\dot{\xinv}_i  = - \NR \xinv_i + (\BI \diag{w} \trasp{\BO} \Phi(\xinv))_i + u_i$.
By assumption~\ref{ass:phimax}, Definition~\eqref{def:hmax} and having proved that $\winv_e(t) \in \mathcal{W}$, for all $t \geq 0$ and $e\in \until{m}$, we get
$\dot{\xinv}_i(t) \leq \umax+ \domax \phimax \wmax- \NR \xinv_i(t)$. Hence, following steps similar to the ones we used to upper bound the synaptic dynamics, we have
\beq\label{eq:x_comparison_lemma}
\abs{\xinv_i(t)} \leq \bigl(\abs{\xinv_i(0)} - \xmax\bigr)\e^{-\NR t} + \xmax,
\eeq
where $\xmax := (\umax+ \domax \phimax \wmax)/\NR$.
Being $\xinv_i(0) \in \mathcal{X}$ we get $\abs{\xinv_i(t)} \leq \xmax$. Hence $\xinv_i(t) \in \mathcal{X}$, for all $t \geq 0$ and for all $i$.
Thus we have that the trajectories of any solution $\trasp{(\trasp{\xinv}, \trasp{\winv})}$ of~\eqref{eq:complete_dynamic_system_nodes_edges} having initial conditions in the set $\mathcal{X}\times\mathcal{W}$ remain in this set, which therefore is forward invariant. Finally, to prove that the set is attractive, we observe that as $t\to+\infty$ conditions~\eqref{eq:w_comparison_lemma} and~\eqref{eq:x_comparison_lemma} are verified for all initial conditions $\trasp{(\trasp{x_0}, \trasp{w_0})}$, not only for that starting in $\mathcal{X}\times\mathcal{W}$. Thus, inequalities~\eqref{in:bounded_evolution_x} and~\eqref{in:bounded_evolution_w} hold and according to Definition~\ref{attractive}, the set $\mathcal{X}\times\mathcal{W}$ is attractive.
\end{proof}
\vspace{-1 mm}
Next, we consider the firing-rate-Hebbian model and show that it exhibits bounded evolution of the solutions if~\ref{ass:phimax} and~\ref{ass:ubarmax} hold.
\begin{lem}[Bounded evolutions firing-rate-Hebbian]\label{lem:bonded_ev_FH}
Consider model~\eqref{eq:firing_rate_complete_dynamic_system_nodes_edges} and let~\ref{ass:phimax},~\ref{ass:ubarmax} hold. Then, the set $\mathcal{V}\times\mathcal{W}$ is forward invariant and attractive, in the sense that, for every neuron $i\in \until{n}$, and every edge $e\in \until{m}$, the following inequalities hold
\begin{align}
\abs{\nu_i(t)} & \leq \bigl(\abs{\nu_i(0)} - \frmax\bigr)\e^{-\NR t} + \frmax, &t \geq 0, \label{eq:bounded_fr1}\\
\abs{w_e(t)}  & \leq \bigl(\abs{w_e(0)}  - \wmax\bigr)\e^{-\SR t} + \wmax, &t \geq 0 \label{eq:bounded_fr2}.
\end{align}
\end{lem}
\begin{proof}
The proof, which follows similar steps to the one given for Lemma~\ref{lem:bonded_ev_HH}, is omitted here for brevity.
\end{proof}
Then, we give the following result for the bounded evolution of the solutions of the Hopfield-Oja model.
\begin{lem}[Bounded evolutions Hopfield-Oja]
\label{lem:bonded_ev_HO}
Consider model~\eqref{eq:complete_dynamic_system_nodes_edges_oja} and let~\ref{ass:phimax},~\ref{ass:umax} and~\ref{ass:ubarmax} hold. Then, the set $\mathcal{X}\times\mathcal{W}$ is forward invariant and attractive in the sense that, for every neuron $i\in\until{n}$, and every edge $e\in\until{m}$, inequalities~\eqref{in:bounded_evolution_x} and~\eqref{in:bounded_evolution_w} hold.
\end{lem}
\begin{proof}
The proof, which follows a reasoning similar to the proof of Lemma~\ref{lem:bonded_ev_HH}, is obtained once the following upper bound for $D^{+}\abs{\winv_e(t)}$ is established
\begin{align*}
D^{+}\abs{\winv_e(t)} \leq \mu\bigl(-\SR I_m - \OR \diag{\BIH}\diag{\BIH}\bigr)\abs{\winv_e(t)} + \hmax \phimax^2 +\ubarmax.
\end{align*}
Applying the {translation property of the log norm (i.e., $\mu\bigl(A+cI_n\bigr) = \mu\bigl(A\bigr) + c, \forall A\in\R^{n\times n}, c\in \R)$} and noticing that $ - \mu\bigl(\OR \diag{\BIH}\diag{\BIH}\bigr)\leq 0$, we get $ D^{+}\abs{\winv_e(t)} \leq \hmax \phimax^2 +\ubarmax -\SR\abs{\winv_e(t)}$, for all $t \geq 0$. The desired results then follow.
\end{proof}
\begin{extendedArxiv}
Finally, we show that firing-rate-Oja model exhibits bounded evolution of the solutions if~\ref{ass:phimax} and~\ref{ass:ubarmax} hold.
\begin{lem}[Bounded evolutions firing-rate-Oja]
\label{lem:bonded_ev_FO}
Consider model~\eqref{eq:firing_rate_complete_dynamic_system_nodes_edges_oja} and let~\ref{ass:phimax},~\ref{ass:ubarmax} hold. Then, the set $\mathcal{V}\times\mathcal{W}$ is forward invariant and attractive, in the sense that, for every neuron $i\in \until{n}$, and every edge $e\in \until{m}$, inequalities~\eqref{eq:bounded_fr1} and~\eqref{eq:bounded_fr2} hold.
\end{lem}
\begin{proof}
The proof, which follows similar steps to the one given for Lemma~\ref{lem:bonded_ev_HO}, is omitted here for brevity.
\end{proof}

\begin{rem}
Lemma~\ref{lem:bonded_ev_HH},~\ref{lem:bonded_ev_FH},~\ref{lem:bonded_ev_HO},~\ref{lem:bonded_ev_FO} ensure that the synaptic rule of each of our models fulfills the boundedness property. This is a desirable property for realistic neural network models, see e.g.~\cite{WG-WK:02}.
\end{rem}
\end{extendedArxiv}
\subsection{Showing Contractivity of the Models}
We now investigate contractivity of the models of Section \ref{modeling}. As we shall see, we give sufficient conditions for strong infinitesimal contractivity (Definition~\ref{def:contracting_system}) by leveraging suitably-defined hierarchical norms. For each model, we propose a contractivity test that depends upon biologically meaningful quantities, such as the neural and the synaptic decay rate, the maximum out-degree, and the maximum synaptic strength. 

Our first result of this Section gives a sufficient condition for the contractivity of the Hopfield-Hebbian model.
\begin{thm}[Strong infinitesimal contractivity of the Hopfield-Hebbian model]
\label{contractivity_hop_heb}
Consider the coupled Hopfield-Hebbian model~\eqref{eq:complete_dynamic_system_nodes_edges} and let~\ref{ass:phimax} --~\ref{ass:ubarmax} hold.
Further, assume that:
\beq\label{eq:conditon_cn_cs}
\NR\SR > 3\domax \hmax\phimax^2 +\domax\ubarmax.
\eeq
 Then, the dynamics~\eqref{eq:complete_dynamic_system_nodes_edges} is strongly infinitesimally contracting on $\mathcal{X}\times\mathcal{W}$ with respect to the norm $\norm{\left[\norm{x}_\infty,\norm{W}_{\infty}\right]}_{p,[\eta]}$, for any $p\in[1,\infty]$ and where $\eta$ is some positive vector. Moreover, the contraction rate is at least
\[
\contrateHH = -\frac{\constHH + \gh -\sqrt{\left(\constHH + \gh \right)^2 - 4\gh\SR^2}}{2\SR},
\]
where $\bmax := \domax \hmax \phimax^2$, $\constHH := \SR^2 + 2\bmax$ , and $\gh := \NR\SR - 3 \bmax-\domax\ubarmax$.
\end{thm}
\begin{proof}
Let us consider the low dimensional formulation of the coupled Hopfield-Hebbian model~\eqref{eq:complete_dynamic_system_nodes_edges} satisfying assumptions~~\ref{ass:phimax} --~\ref{ass:ubarmax}.
Its Jacobian is
$
J(x,w) :=
\left[
\begin{array}{c|c}
\subscr{J}{nn} & \subscr{J}{ns} \\
\hline
\subscr{J}{sn} & \subscr{J}{ss}
\end{array}
\right],
$
where, defining $\subscr{f}{n} := - \NR x + \BI \diag{w} \trasp{\BO} \Phi(x) + u$, and $\subscr{f}{s} := \hh \circ \BOH \circ \BIH - \SR w + \bar u$, we have 
\begin{align*} 
\subscr{J}{nn} &= \derp{\subscr{f}{n}}{x}= - \NR I_n+ \BI \diag{w} \trasp{\BO} \diag{\Phi'(x)},\\
\subscr{J}{ns} &= \derp{\subscr{f}{n}}{w} = \derp{\bigl(\BI \diag{\trasp{\BO} \Phi(x)}w\bigr)}{w} = \BI \diag{\trasp{\BO} \Phi(x)},\\
\subscr{J}{sn} &= \derp{\subscr{f}{s}}{x}= \diag{\hh}\bigl(\diag{ \BOH}\trasp{\BI} + \diag{\BIH}\trasp{\BO} \bigr) \diag{\Phi'(x)},\\
\subscr{J}{ss} &= \derp{\subscr{f}{s}}{w}= - \SR I_m.
\end{align*}
We consider the infinity norm both on $\R^n$ and $\R^m$ and we define the aggregate Metzler majorant of the matrix $J(x,w)$:
\beq\label{aggregate_Metzler_majorant}
\JacM =
\left[
\begin{array}{cc}
\mu_{\infty}(\subscr{J}{nn}) & \norminf{\subscr{J}{ns}} \\
\norminf{\subscr{J}{sn}} & \mu_{\infty}(\subscr{J}{ss})
\end{array}
\right].
\eeq
{For any $p\in[1,\infty]$ and for $\eta>0$ as in Lemma~\ref{lem:eta} we consider the aggregation norm $\|\cdot\|_{\agg} = \|\cdot\|_{p,[\eta]}$.}
From Theorem~\ref{thm:composite-norm-majorants} we get that for any $p\in[1,\infty]$ and for $\eta>0$ we have
\[
\lognorm{J(x,W)}{\composite} \leq \lognorm{\JacM}{p,[\eta]}.
\]

From Proposition~\ref{prop:BI_BO} it follows $\norminf{\BI} = \domax$ and $\norminf{\trasp{\BI}} = \norminf{\trasp{\BO}}=1$. Also, being ${\norminf{\diag{\Phi'(x)}}\leq1}$, $\norminf{\Phi(x)} = \phimax$ and $|W_{ij}| \leq (\hmax \phimax^2 + \ubarmax)/\SR$, $\forall i,j\in\until{n}$, we upper bound:
\begin{align*}
\norminf{\subscr{J}{ns}}
&\leq \norminf{\BI} \norminf{\diag{\BOH}} \leq \domax \phimax,\\
\norminf{\subscr{J}{sn}}& \leq \norminf{\diag{\hh}} \norminf{\diag{\trasp{\BO}\Phi(x)}\trasp{\BI} + \diag{\trasp{\BI}\Phi(x)}\trasp{\BO}} \leq 2\hmax\phimax,\\
\mu_{\infty}(\subscr{J}{nn}) 
&= -\NR + \mu_{\infty}\bigl(\BI \diag{w} \trasp{\BO} \diag{\Phi'(x)}\bigr) \leq  - \NR + \norminf{\BI \diag{w} \trasp{\BO} \diag{\Phi'(x)}}\\
&\leq  \domax \bigl(\hmax \phimax^2 + \ubarmax\bigr)/\SR - \NR,\\ \mu_{\infty}(\subscr{J}{ss})
&= -\SR.
\end{align*}
Hence:
\begin{align*}
\JacM &\leq
\begin{bmatrix}
\displaystyle \frac{\domax \bigl(\hmax \phimax^2 + \ubarmax\bigr)}{\SR}-\NR & \domax \phimax \\
2\hmax\phimax & -\SR
\end{bmatrix} := \JacMTildeHH.
\end{align*}
{Applying the monotonicity property of the log norm of a Metzler matrix, and being $\JacMTildeHH$ an irreducible Metzler matrix, from Lemma~\ref{lem:eta} we have:}
$
\lognorm{\JacM}{p,[\eta]} \leq \lognorm{\JacMTildeHH}{p,[\eta]} = \alpha(\JacMTildeHH).
$

Finally, the last step is to find conditions for which the matrix $\JacMTildeHH$, and thus $\JacM$, is Hurwitz. In our case, being $\JacMTildeHH$ a $2\times 2$ matrix this happens if and only if
\begin{align} \label{eq:cond_hurwitz_det}
\det{(\JacMTildeHH)} = \NR\SR - 3\domax\hmax\phimax^2 -\domax\ubarmax>0,
\end{align}
and $\tr(\JacMTildeHH)= \domax (\hmax \phimax^2 + \ubarmax)/\SR - \NR - \SR <0$, i.e.,
\beq\label{eq:cond_hurwitz_trace}
\NR\SR > -\SR^2 + \domax(\hmax \phimax^2 + \ubarmax).
\eeq
Now, since condition~\eqref{eq:cond_hurwitz_det} implies~\eqref{eq:cond_hurwitz_trace}, $\JacMTildeHH$ is Hurwitz if and only if condition~\eqref{eq:cond_hurwitz_det}, i.e., condition~\eqref{eq:conditon_cn_cs}, is verified and computing the spectral abscissa yields (see Appendix)
\beq\label{spectral_abscissa}
\begin{split}
\alpha(&\JacMTildeHH) :=\contrateHH = \frac{ \sqrt{\left(\constHH + \gh \right)^2 - 4\gh\SR^2}-\constHH - \gh}{2\SR},
\end{split}
\eeq
where $\bmax := \domax \hmax \phimax^2$, $\constHH := \SR^2 + 2\bmax$ , and $\gh := \NR\SR - 3 \bmax-\domax\ubarmax$.
Hence, if \eqref{eq:conditon_cn_cs} is satisfied, then, from Definition~\ref{def:contracting_system}, we have that the coupled neural synaptic dynamics~\eqref{eq:complete_dynamic_system_nodes_edges} is strongly infinitesimally contracting and its contraction rate is at least $\contrateHH$. This proves the result.
\end{proof}

With the next result, we give a sufficient condition for the contractivity of the firing-rate-Hebbian model.

\begin{thm}[Strong infinitesimal contractivity of the firing-rate-Hebbian model]
Consider model~\eqref{eq:firing_rate_complete_dynamic_system_nodes_edges} and let~\ref{ass:phimax},~\ref{ass:phiprime} and~\ref{ass:ubarmax} hold.
Further, assume that:
\beq\label{eq:conditon_cn_cs_fr}
\NR\SR > \domax \hmax\phimax^2\Bigl(1 + \frac{2}{\NR}\Bigr) +\domax\ubarmax.
\eeq
Then, the dynamics~\eqref{eq:firing_rate_complete_dynamic_system_nodes_edges} is strongly infinitesimally contracting on $\mathcal{V}\times\mathcal{W}$ with respect to the norm $\norm{\left[\norm{\nu}_\infty,\norm{W}_{\infty}\right]}_{p,[\eta]}$, for any $p\in[1,\infty]$ and where $\eta$ is some positive vector. Moreover, the contraction rate is at least
\beq\label{spectral_abscissa_fr}
\contrateFH = -\frac{\constFH + \gfr - \sqrt{\left(\constFH + \gfr \right)^2 - 4\gfr\SR^2}}{2\SR},
\eeq
where $\displaystyle \gfr := \NR\SR - \amax\left( 1 + {2}/{\NR}\right)- \ubarmax\domax$ and $\displaystyle \constFH := \SR^2 + {2\amax}/{\NR}$.
\end{thm}
\begin{proof}
The proof follows similar steps as these used to prove Theorem~\ref{contractivity_hop_heb}. The full proof is therefore omitted here, we only notice that, for the firing-rate-Hebbian dynamics, the Jacobian is partitioned into the following matrices:
\begin{align*}
\subscr{J}{nn} &= \BI\diag{w}\trasp{\BO}\left[\Phi'\bigl(\BI\diag{w}\trasp{\BO}\fr+u\bigr)\right]-\NR I_n,\\
\subscr{J}{ns} &= \BI \diag{\trasp{\BO}\fr} \left[\Phi'\bigl(\BI \diag{w} \trasp{\BO} \fr + u\bigr)\right],\\
\subscr{J}{sn} &= \diag{\hh}\bigl(\diag{ \BOF}\trasp{\BI} + \diag{\BIF}\trasp{\BO} \bigr) \diag{\Phi'(\nu)},\\
\subscr{J}{ss} &= - \SR I_m.
\end{align*}
\end{proof}
\begin{extendedArxiv}
Nevertheless, for completeness, we present the Jacobian of~\eqref{eq:firing_rate_complete_dynamic_system_nodes_edges} in Appendix~\ref{appx:jacobian_fr} and the computation of the spectral abscissa in Appendix~\ref{apx:spectral_abscissa_fr}.
\end{extendedArxiv}

With the next result, we give a sufficient condition for the contractivity of the Hopfield-Oja model.
\begin{thm}[Strong infinitesimal contractivity of the Hopfield-Oja model]
Consider the dynamics~\eqref{eq:complete_dynamic_system_nodes_edges_oja} and let~\ref{ass:phimax} --~\ref{ass:ubarmax} hold. Further, assume that:
\beq\label{eq:conditon_cn_cs_oja}
\begin{split}
\NR\SR > \domax\bigl(3\hmax\phimax^2 +\ubarmax\bigr) + 2 \frac{\OR}{\SR}\phimax^2\domax\bigl(\hmax\phimax^2 + \ubarmax\bigr).
\end{split}
\eeq
Then, the dynamics~\eqref{eq:complete_dynamic_system_nodes_edges_oja} is strongly infinitesimally contracting on $\mathcal{X}\times\mathcal{W}$ with respect to the norm $\norm{\left[\norm{x}_\infty,\norm{W}_{\infty}\right]}_{p,[\eta]}$, for any $p\in[1,\infty]$ and where $\eta$ is some positive vector. Moreover, the contraction rate is at least
\[
\contrateHO =
-\frac{\constHO + \go -\sqrt{\left(\constHO + \go\right)^2 - 4\go\SR^2}}{2\SR},
\]
where $\displaystyle \constHO := \SR^2 +2\bmax +2\phi^2\OR/\SR\bigl(\bmax+\domax\ubarmax\bigr)$ and $\displaystyle \go := \NR\SR -3\domax\hmax\phimax^2 +\domax\ubarmax - 2 \frac{\OR}{\SR}\phimax^2 \bigl(\domax\hmax\phimax^2 + \domax\ubarmax\bigr)$.
\end{thm}
\begin{proof}
The full proof, which follows similar steps to the ones used to prove  Theorem~\ref{contractivity_hop_heb}, is omitted here for brevity. We note that, for the Hopfield-Oja dynamics, the Jacobian is partitioned into the following matrices: 
\begin{align*}
\subscr{J}{nn} &= - \NR I_n + \BI \diag{w} \trasp{\BO} \diag{\Phi'(x)},\\
\subscr{J}{ns} &= \BI \diag{\trasp{\BO} \Phi(x)},\\
\subscr{J}{sn} &= \diag{\hh}\bigl(\diag{ \trasp{\BO}\Phi(x)}\trasp{\BI} + \diag{\trasp{\BI}\Phi(x)}\trasp{\BO}\bigr)\diag{\Phi'(x)}-2\OR\diag{w}\diag{\BIH}\trasp{\BI} \diag{\Phi'(x)},\\
\subscr{J}{ss} &= - \SR I_m - \OR \diag{\BIH}\diag{\BIH}.
\end{align*}
\end{proof}

\begin{extendedArxiv}
The main differences are the the Jacobian~\eqref{eq:complete_dynamic_system_nodes_edges_oja} and the computation of the spectral abscissa, that, for completeness, we present in  Appendix~\ref{appx:jacobian_ho} and~\ref{apx:spectral_abscissa_hopfield_oja}, respectively.
\end{extendedArxiv}
\begin{extendedArxiv}
Finally, with the next result, we give a sufficient condition for the contractivity of the firing-rate-Oja model.
\begin{thm}[Strong infinitesimal contractivity of the firing-rate-Oja model]
Consider model~\eqref{eq:firing_rate_complete_dynamic_system_nodes_edges_oja} and let~\ref{ass:phimax},~\ref{ass:phiprime} and~\ref{ass:ubarmax} hold.
Further, assume that:
\beq\label{eq:conditon_cn_cs_fr_oja}
\NR\SR > \domax \hmax\phimax^2\Bigl(1 + \frac{2}{\NR}\Bigr) +\domax\ubarmax + 2 \frac{\OR}{\SR\NR}\phimax^2\domax\bigl(\hmax\phimax^2 + \ubarmax\bigr).
\eeq
Then, the dynamics~\eqref{eq:firing_rate_complete_dynamic_system_nodes_edges_oja} is strongly infinitesimally contracting on $\mathcal{V}\times\mathcal{W}$ with respect to the norm $\norm{\left[\norm{\nu}_\infty,\norm{W}_{\infty}\right]}_{p,[\eta]}$, for any $p\in[1,\infty]$ and where $\eta$ is some positive vector. Moreover, the contraction rate is at least
\[
\contrateFO = -\frac{\constFO + \gof-\sqrt{\left(\constFO + \gof\right)^2 - 4\gof\SR^2}}{2\SR},
\]
where $\amax := \domax \hmax \psimax^2$, $\constFO := \SR^2 + 2\amax/\NR +2\phi^2\frac{\OR}{\SR\NR}\bigl(\amax+\domax\ubarmax\bigr)$ and $\gof = \NR\SR -\amax\Bigl(1 + \frac{2}{\NR}\Bigr) -\domax\ubarmax -2\frac{\OR}{\SR\NR}\phimax^2\bigl(\amax + \domax\ubarmax\bigr)$.
\end{thm}
Again we omit the proof since it follows the same steps of Theorem~\ref{contractivity_hop_heb}. The main differences are the the Jacobian of the system~\eqref{eq:firing_rate_complete_dynamic_system_nodes_edges_oja} and the computation of the spectral abscissa, that, for completeness, we present in  Appendix~\ref{appx:jacobian_fro} and~\ref{apx:spectral_abscissa_firingrate_oja}, respectively.
\end{extendedArxiv}

Finally, we close this section with the following observation. Comparing~\eqref{eq:conditon_cn_cs},~\eqref{eq:conditon_cn_cs_fr},~\eqref{eq:conditon_cn_cs_oja} and~\eqref{eq:conditon_cn_cs_fr_oja}, we can see that the contractivity test for~\eqref{eq:complete_dynamic_system_nodes_edges_oja} and~\eqref{eq:firing_rate_complete_dynamic_system_nodes_edges_oja} are more conservative than the test for~\eqref{eq:complete_dynamic_system_nodes_edges} and~\eqref{eq:firing_rate_complete_dynamic_system_nodes_edges}.
On the other hand, when $\NR = 1$ the contractivity test for the Hopfield-Hebbian and the firing-rate-Hebbian models is the same, while when $c_n >1$~\eqref{eq:conditon_cn_cs_fr} gives sharper contractivity condition with respect to~\eqref{eq:conditon_cn_cs}. Vice versa when $c_n <1$.
\begin{rem}
Sparse connectivity enables the low-dimensional reformulations of systems~\eqref{eq:complete_dynamic_system} -- \eqref{eq:complete_dynamic_system_fr_oja} as systems~\eqref{eq:complete_dynamic_system_nodes_edges} -- \eqref{eq:firing_rate_complete_dynamic_system_nodes_edges_oja}. Conditions~\eqref{eq:conditon_cn_cs},~\eqref{eq:conditon_cn_cs_fr},~\eqref{eq:conditon_cn_cs_oja} and~\eqref{eq:conditon_cn_cs_fr_oja}, are less conservative compared to what one could obtain if the analysis was applied directly to the original dynamics~\eqref{eq:complete_dynamic_system} -- \eqref{eq:complete_dynamic_system_fr_oja}.
\end{rem}
\begin{rem}
Throughout the paper, as in e.g.,~\cite{MG-OF-PB:12, WG-WK:02, LK-ME-JJES:22, KDM-FF:12}, we consider homogeneous decay rates. However, it is worth noting that our analysis can be generalized to heterogeneous decay rates. For example, let $c_n^i$ be the decay rate for the $i$-th neuron so that the dynamics~\eqref{eq:hopfield_dynamic} reads
$$
\dot x_i = - \NR^i x_i + \sum_{j=1}^n W_{ij}\phi(x_j) + u_i.
$$
Then the contractivity condition~\eqref{eq:conditon_cn_cs} becomes
$$
\big(\min_{i}\NR^i\big)\SR > 3\domax \hmax\phimax^2 + \domax\ubarmax.
$$
\end{rem}

\subsection{Invariance Results for the Synaptic Dynamics}\label{Invariance results for the synaptic dynamics}
Biological neurons release either excitatory (E) or inhibitory (I) outgoing synapses, not both~\cite{HD:35, JCE-PF-KK:54}.
This property, known as \emph{Dale’s Principle}, implies that neurons cannot have a mixture of positive and negative output synapses, and furthermore, the inhibitory/excitatory nature of the synapses cannot change over time. This means that the elements of the columns of the synaptic matrix $W$ are either all non-negative or non-positive $\forall t$.
We now investigate if the models considered in this paper satisfy Dale's Principle.

\begin{lem}[Dale's Principle]\label{lem:invariance_results_for_the_synaptic_dynamics}
Consider the models~\eqref{eq:complete_dynamic_system} -- \eqref{eq:complete_dynamic_system_fr_oja} with external synaptic stimuli $\bar U = 0$ and let~\ref{ass:phimax}--\ref{ass:umax} hold. Pick a neuron $j \in \until{n}$. If, for all $i \in \until{n}$:
\begin{enumerate}
\item \label{stm:dale>0}
$H_{ij} > 0$ and $W_{ij}(0) \geq 0$, then $W_{ij}(t) \geq 0$, $\forall t \geq0$;
\item \label{stm:dale<0}
$H_{ij} < 0$ and $W_{ij}(0) \leq 0$, then $W_{ij}(t) \leq 0$, $\forall t \geq0$.
\end{enumerate}
\end{lem}
\begin{proof}
We start by proving part~\ref{stm:dale>0} for the dynamics in~\eqref{eq:complete_dynamic_system_oja} with $\bar U_{ij} = 0$. We show the result by considering the synaptic dynamics
\begin{align}
\label{eq:synaptic_dynamics_no-input}
\dot W_{ij}(t) &= \HH_{ij}\phi(y_i(t))\phi(y_j(t)) - (\SR + \OR \phi^2(y_i(t))) W_{ij}(t),
\end{align}
with $y_i(t)$ and $y_j(t)$ being exogenous inputs. Let $W_j := (W_{ij})_{i \in \until{n}} \in \R^n$ be the $j$-th column of the matrix $W$. We show that, if the assumptions in (i) are satisfied, then the positive orthant is forward invariant for the dynamics for $W_j$ uniformly in $y_i(t)$ and $y_j(t)$.

To this aim, note that by assumptions when $W_{ij}=0$ the right hand side in~\eqref{eq:synaptic_dynamics_no-input} is non-negative. Hence, by Nagumo's Theorem~\cite{MN:1942}, the positive orthant is forward invariant for~\eqref{eq:synaptic_dynamics_no-input}. Moreover, since this property holds for all signals $y_i(t)$ and $y_j(t)$, this property also holds when $y_i(t) = x_i(t)$ and $y_j(t) = x_j(t)$. This gives the result for~\eqref{eq:complete_dynamic_system_oja} and \eqref{eq:complete_dynamic_system_fr_oja}. Furthermore, the non-negativity condition for the right hand side in~\eqref{eq:synaptic_dynamics_no-input} also holds when $\OR = 0$ and this in turn yields the result for models~\eqref{eq:complete_dynamic_system} and~\eqref{eq:firing_rate_complete_dynamic_system}.
The proof for part~\ref{stm:dale<0} follows similar reasoning and is omitted here for brevity.
\end{proof}

\begin{rem}
A key assumption in Lemma~\ref{lem:invariance_results_for_the_synaptic_dynamics} is that the models satisfy~\ref{ass:phimax}, ensuring the activation function's non-negativity. 
It is worth noting that if the activation function acting on the pre-synaptic node $y_j$ has an opposite sign with respect to the one acting on the post-synaptic node $y_i$, then not only Dale's principle is not satisfied, but neurons over time will change the outgoing synapses they release.
That is, excitatory synapses become inhibitory and viceversa.
\end{rem}

Finally, we investigate invariant results for symmetric synaptic matrices. We analyze this aspect as in the neuroscience literature this appears to be a key property for a number of well known models, e.g.,~\cite{DWD-JJH:92, JJH:84, LK-ML-JJES-EKM:20, BS-YB:17}. Nevertheless, the assumption of symmetric weight matrices, which is often made to streamline the mathematical analysis, violates Dale’s Principle. To investigate invariant results for symmetric synaptic matrices, we consider the dynamics~\eqref{eq:hebbian_dynamic} in vector form with $\bar U = 0$:
\beq \label{eq:synaptic_dynamics_no-input_vector_form}
\dot W(t) = \HH \circ \Phi(y(t))\trasp{\Phi(y(t))} - \SR W(t),
\eeq
where $y(t)$ is an exogenous input.
The following Lemma formalizes the fact that, if $H$ is symmetric the system always converges to a symmetric synaptic matrix. Moreover, if $W(0)$ is symmetric, then $W(t)=\trasp{W}(t)$, $\forall t \ge 0$.
\begin{lem}\label{lem:symmetric_matrix}
Consider the models~\eqref{eq:complete_dynamic_system} and~\eqref{eq:firing_rate_complete_dynamic_system} with external synaptic stimuli $\bar U = 0$ and let~\ref{ass:phimax}--\ref{ass:umax} hold. Assume that $\HH$ is symmetric and let $\subscr{W}{S}(t)$ and $\subscr{W}{A}(t)$ be the symmetric and skew-symmetric components of $W(t)$, then:
\begin{enumerate}
\item \label{stm:symmetric_1}
if $\subscr{W}{A}(0) = 0$, then $\subscr{W}{A}(t) = 0 $, $\forall t\geq 0;$
\item \label{stm:symmetric_2}
$\displaystyle \lim_{t \to \infty}  \subscr{W}{A}(t) = 0.$
\end{enumerate}
\end{lem}
\begin{proof}
The proof is inspired by~\cite[Appendix 8.1]{MG-OF-PB:12}. First, we write $W(t) = \subscr{W}{S}(t) + \subscr{W}{A}(t)$, $\forall t \geq 0$, so that equation~\eqref{eq:synaptic_dynamics_no-input_vector_form} can be written as
$$
\subscr{\dot W}{S} + \subscr{\dot W}{A} = (\HH \circ \Phi(y)\trasp{\Phi(y)} - \SR \subscr{W}{S}) - \SR\subscr{W}{A}.
$$
To prove part~\ref{stm:symmetric_1} note that $\subscr{W}{A}(0) = 0$ implies that the right end side in~\eqref{eq:synaptic_dynamics_no-input_vector_form} is symmetric at time $t=0$, thus $\subscr{W}{A}(t) = 0 $, $\forall t\geq 0$ and uniformly in $y(t)$. This leads to the desired result. Next, for part~\ref{stm:symmetric_2} it suffices to note that the dynamics for the skew-symmetric component of $W(t)$ are given by $\subscr{\dot W}{A}(t) = - \SR \subscr{W}{A}(t)$, whose solution is $\subscr{W}{A}(t) = \subscr{W}{A}(0)\e^{- \SR t}$, $\forall t {\geq 0}$. The
desired result then follow.
\end{proof}
\begin{rem}
Remarkably the previous results hold only when $\OR = 0$. In fact, if $\OR \neq 0$, in general the right end side of~\eqref{eq:synaptic_dynamics_no-input} is not symmetric and therefore Lemma~\ref{lem:symmetric_matrix} can't apply.
\end{rem}

\section{Numerical Example}\label{Numerical Example}
We validate our theoretical results via a simple example and, for brevity, we present numerical results only for the Hopfield-Hebbian model~\eqref{eq:complete_dynamic_system_nodes_edges}.
Inspired by one of the building blocks of the nematode C.~elegans neural circuit studied in~\cite{NM-AMZ-TE:22}, we consider the simple network of Figure~\ref{fig:num_example} with six neurons and six edges (four excitatory and two inhibitory). The C.~elegans architecture, in fact, can be schematically represented as a cascade of the blocks of the network in Figure~\ref{fig:num_example}.
\begin{figure}[!ht]
\centerline{\includegraphics[scale=.45]{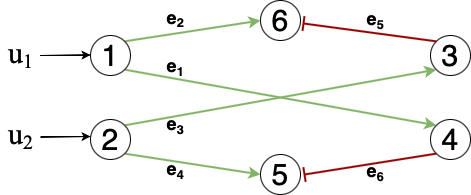}}
\caption{Coupled neural-synaptic model with six neurons $i$ and six edges, $e_i$, $i=1,\dots,6$, four excitatory (green) and two inhibitory (red). Only nodes $1$ and $2$ are subjected to the external stimuli $u_1$ and $u_2$, respectively. Colors online.}
\label{fig:num_example}
\end{figure}

\noindent
The out- and in-incidence matrices for this network are:
\[
\BI = 
\begin{bmatrix}
0 & 0 & 0 & 0 & 0 & 0\\
0 & 0 & 0 & 0 & 0 & 0\\
0 & 0 & 1 & 0 & 0 & 0\\
1 & 0 & 0 & 0 & 0 & 0\\
0 & 0 & 0 & 1 & 0 & 1\\
0 & 1 & 0 & 0 & 1 & 0
\end{bmatrix},\ 
\BO =
\begin{bmatrix}
1 & 1 & 0 & 0 & 0 & 0\\
0 & 0 & 1 & 1 & 0 & 0\\
0 & 0 & 0 & 0 & 1 & 0\\
0 & 0 & 0 & 0 & 0 & 1\\
0 & 0 & 0 & 0 & 0 & 0\\
0 & 0 & 0 & 0 & 0 & 0
\end{bmatrix}
\]
In this case $\domax = 2$ and we pick {the elements of $h$ in~\eqref{eq:complete_dynamic_system_nodes_edges} from the interval $[-1,1]$}. These elements are selected so that $e_i$, $i \in \until{4}$, are excitatory, while $e_5$ and $e_6$ are inhibitory.
For the neurons we set $\phi(x) = \frac{1}{1 + \e^{-x}}$ and hence $\phimax = 1$.
In the network, only neurons $1$ and $2$ receive the external stimuli $u_1 = 20 \sin(8t)$ and $u_2 = 15 \cos(8t)$, respectively; also, the excitatory synaptic weights are subject to a constant stimulus $\bar u = 1.5$. In our experiments, we pick the initial conditions from the set $[-1, 1]$ and select the synaptic initial conditions so that Dale's Principle is satisfied.
In order to numerically validate {the results presented in Section~\ref{Dynamical properties of the model}} we set $c_n = 3.6$ and $c_s = 3.2$, so that condition~\eqref{eq:conditon_cn_cs} is satisfied, i.e., the Hopfield-Hebbian network is strongly infinitesimally contracting. The behavior of the network is illustrated in Figure~\ref{fig:contracting_dynamics}.
The contraction rate estimate given by~\eqref{spectral_abscissa} is $\contrateHH = 0.54$. As expected, this estimate is more conservative than the empirical contraction rate of $4.10$ obtained from numerical simulations.
\begin{figure}[!h]
    \centering
    \includegraphics[scale = 0.68]{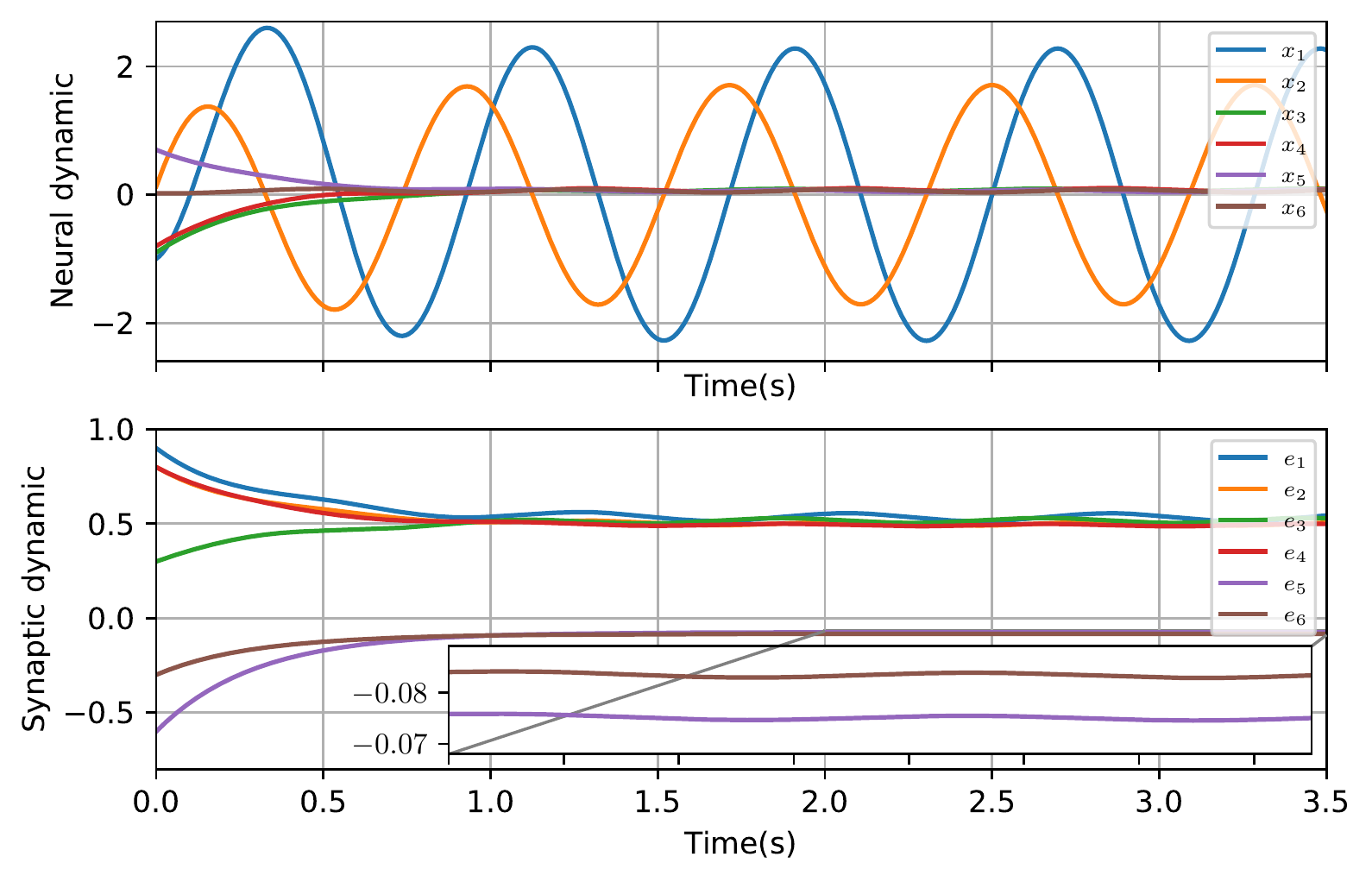}
    \caption{Simulation of the Hopfield-Hebbian model of Figure~\ref{fig:num_example} exhibiting entrainment to periodic inputs   typical of contracting systems. See Section~\ref{Numerical Example} for the parameters.}
    \label{fig:contracting_dynamics}
\end{figure}
Moreover, a direct computation shows that $\xmax = 5.98$ and $\wmax = 0.78$, in accordance with Lemma~\ref{lem:bonded_ev_HH}. Also, the behavior in the figure is in accordance with Lemma~\ref{lem:invariance_results_for_the_synaptic_dynamics}: note indeed that the synaptic weights have always the same sign. We also note that, since our conditions guarantee contractivity of the network, the Hopfield-Hebbian network becomes entrained by the periodic inputs $u_1$ and $u_2$ (see Figure~\ref{fig:contracting_dynamics}).
\begin{figure}[!ht]
\centerline{\includegraphics[scale=.45]{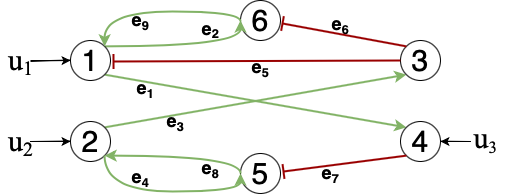}}
\caption{Coupled neural-synaptic model with six neurons $i$, $i=1,\dots,6$, and nine edges, $e_j$, $j=1,\dots,9$, six excitatory (green) and three inhibitory (red). Only nodes $1$, $2$, and $4$ are subjected to the external stimuli $u_1$, $u_2$, and $u_3$, respectively. Colors online.}
\label{fig:recurrent_network}
\end{figure}

Finally, to validate our results for RNNs, we introduce recurrent connections to the network in Figure~\ref{fig:num_example}, obtaining the network in Figure~\ref{fig:recurrent_network}. Here, $\domax = 2$ and we pick {the elements of $h$ in~\eqref{eq:complete_dynamic_system_nodes_edges} from the interval $[-1,1]$}. These elements are selected so that $e_5$, $e_6$, and $e_7$ are inhibitory, while the other edges are excitatory.
In the network, only neurons $1$, $2$, and $4$ receive the external stimuli $u_1 = 5\tanh{(t)}$, $u_2 = 3\tanh{(t)}$, and $u_3 = 7\tanh{(t)}$, respectively; also, the synaptic weights $e_1$ and $e_4$ are subject to a constant stimulus $\bar u = 1.5$, while $e_2$, $e_3$ and $e_9$ to $\bar u = 1$. In our experiments, we pick the initial conditions from the set $[-1, 1]$ and select the synaptic initial conditions so that Dale's Principle is satisfied. Also in this case, we set $c_n = 3.6$ and $c_s = 3.2$, so that condition~\eqref{eq:conditon_cn_cs} is satisfied.
For this example, we perform an exploratory numerical study to investigate what happens when the network parameters are set so to satisfy {Theorem~\ref{contractivity_hop_heb}}, but the activation functions are affected by the delay, say $\tau = 2s$ in our simulations. While it is well known that contraction is preserved through specific time-delayed communications~\cite{WW-JJES:06b}, {to the best of our knowledge this property has not been investigated for the types of dynamics considered here.} The resulting behavior of the network, illustrated in Figure~\ref{fig:contracting_delay}, shows that the delayed system appears to be still contracting. We leave the study of neural-synaptic networks with delays to future work.
\begin{figure}[!h]
   \centering
    \includegraphics[scale = 0.68]{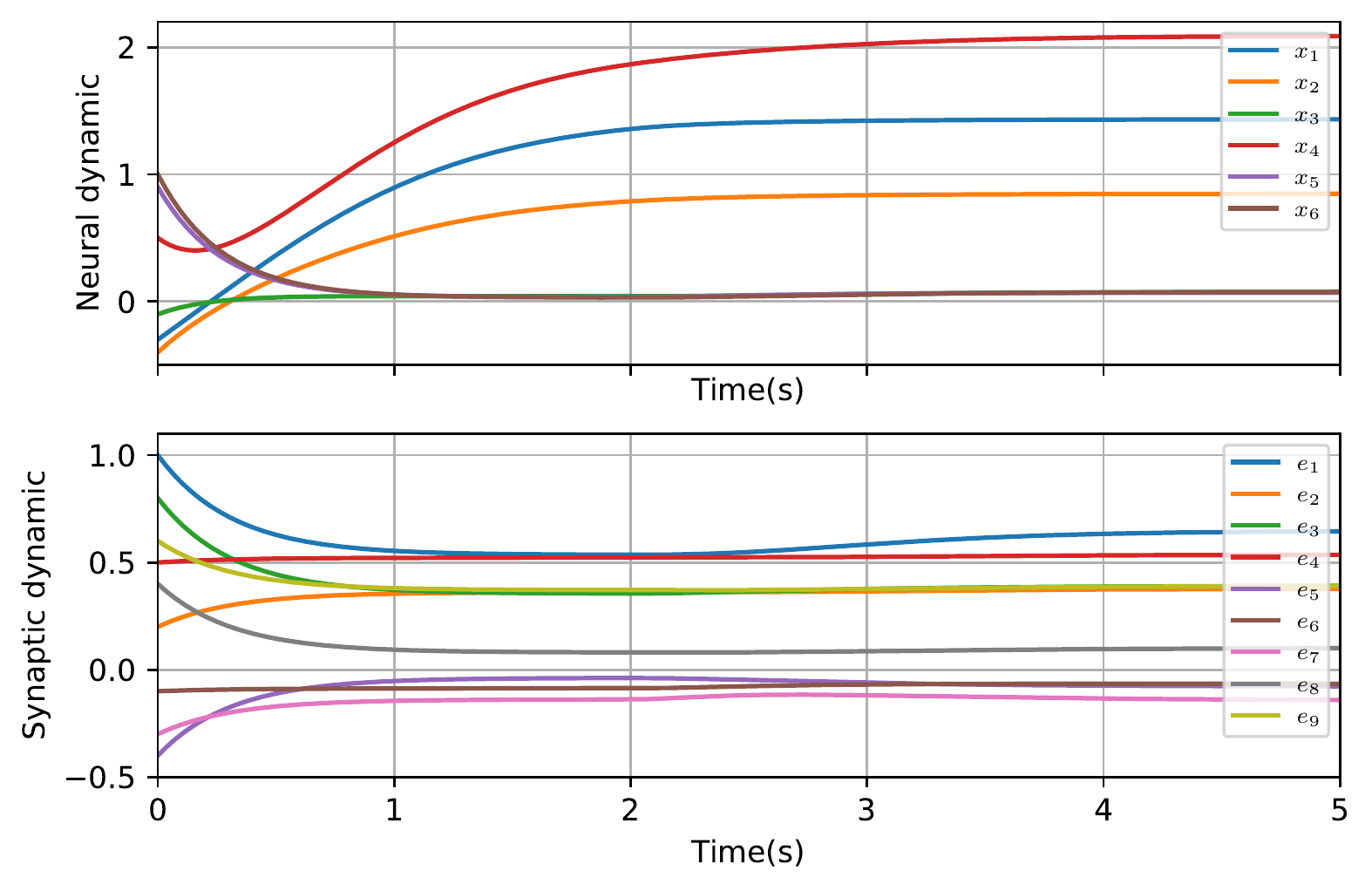}
    \caption{Simulation of the Hopfield-Hebbian model of Figure~\ref{fig:recurrent_network} with the activation functions affected by $2s$ of delay. Even with the delays the system appears to be still contracting. See Section~\ref{Numerical Example} for the parameters.}
    \label{fig:contracting_delay}
\end{figure}

\section{Conclusion and Future Work} \label{Conclusion}
We presented the modeling and analysis of four coupled neural-synaptic models: the Hopfield-Hebbian model, the firing-rate-Hebbian model, the Hopfield-Oja model, and the firing-rate-Oja model. We considered networks with both excitatory and inhibitory synapses governed by both Hebbian and anti-Hebbian rules.
In~\cite{KDM-FF:12} it is shown that when the synaptic dynamics are constant under proper transformations the models~\eqref{eq:hopfield_dynamic} and~\eqref{eq:firing_rate_dynamic} are mathematically equivalent. In general, this is not true when the matrix $W$ is time dependent, which is a key assumption of our models. Therefore we analyzed both neural models~\eqref{eq:hopfield_dynamic} and~\eqref{eq:firing_rate_dynamic}.
To capture the synaptic sparsity of neural circuits, for each model we proposed a low dimensional modeling formulation that allowed us to go from a system with $n\times n^2$ variables--$n$ neurons and $n^2$ synaptic connections--to a system with $n\times m$ variables, where $m\ll n^2$ is the number of non zero elements of the synaptic connection matrix $H$.
We then characterized the key dynamical properties of the models. First, we gave a biologically-inspired forward invariance result for the trajectories of the system.
Then, as key result, we gave sufficient conditions for the non-Euclidean contractivity of the models. Each contractivity test we presented is based upon biologically meaningful quantities, i.e., neural and synaptic decay rate, maximum in-degree, and the maximum synaptic strength. Particularly, we found that when the neural decay rate $\NR >1$ the model with the FRNN has sharper contractivity conditions with respect to the one with the HNN. Finally, we showed that under suitable conditions the synaptic rules satisfy Dale's Principle and illustrated the effectiveness of our results via a numerical example.
Motivated by the numerical findings reported in Figure~\ref{fig:contracting_delay} in our future work we plan to consider models with delays. We will also explore the functional implications of the HL rules considered here.
A possible future research direction could be the extension of our analysis to bounded confidence models in opinion dynamics~\cite{RH-UK:02}, given their analogies with the time-dependent neural dynamics investigated in this paper. Finally, it would be interesting to explore how to use the contractivity results obtained in this paper to provide robustness guarantees of neural networks in machine learning.

\begin{extendedArxiv}
\appendix
\subsection{Computation of the Jacobian}
For completeness, we present here the Jacobian of the firing-rate-Hebbian model, the Hopfield-Oja model, and the firing-rate-Oja model (the one of the Hopfield-Hebbian model can be found in the proof of Theorem~\ref{contractivity_hop_heb}).

To this purpose for every model, we consider the low dimensional formulation in $n\times m$ variables. We define the vector $f = [\subscr{f}{n}(y,w),\ \subscr{f}{s}(y,w)]^{\textsf{T}} \in \R^{n+m}$, where $\subscr{f}{n}(y,w)$ is the function describing the neural dynamic, $\subscr{f}{s}(y,w)$ is the function describing the synaptic dynamics, $y$ is the membrane potential in the case of the Hopfield network or the firing-rate in the so-called model, and $w$ is the synaptic vector.
The Jacobian of coupled neural-synaptic  system is:
\[
J(y,w) = 
\left[
\begin{array}{c|c}
\displaystyle \derp{\subscr{f}{n}}{y} & \displaystyle \derp{\subscr{f}{n}}{w}\vspace{0.3 mm}\\
\hline 
\displaystyle \derp{\subscr{f}{s}}{y} & \displaystyle \derp{\subscr{f}{s}}{w}
\end{array}
\right]
:= 
\left[
\begin{array}{c|c}
\subscr{J}{nn} & \subscr{J}{ns} \\
\hline
\subscr{J}{sn} & \subscr{J}{ss}
\end{array}
\right].
\]
\subsubsection{Firing-rate-Hebbian Model}\label{appx:jacobian_fr}
In this case we define $\subscr{f}{n} := - \NR \fr + \Phi\bigl(\BI \diag{w} \trasp{\BO} \fr + u\bigr)$ and $\subscr{f}{s} :=\hh \circ \bigl( \BOF\bigr)\circ \left(\BIF\right) - \SR w  + \bar u$. We have:
\begin{align*}
\subscr{J}{nn} &= - \NR I_n+ \BI \diag{w} \trasp{\BO} \left[\Phi'\bigl(\BI \diag{w} \trasp{\BO} \fr + u\bigr)\right],\\
\subscr{J}{ns} &= \BI \diag{\trasp{\BO}\fr} \left[\Phi'\bigl(\BI \diag{w} \trasp{\BO} \fr + u\bigr)\right],\\
\subscr{J}{sn} &= \diag{\hh}\bigl(\diag{ \BOF}\trasp{\BI} + \diag{\BIF}\trasp{\BO} \bigr) \diag{\Phi'(\nu)},\\
\subscr{J}{ss} &= - \SR I_m.
\end{align*}

\subsubsection{Hopfield-Oja Model}\label{appx:jacobian_ho}
It is $ \subscr{f}{n} := - \NR x + \BI \diag{w} \trasp{\BO} \Phi(x) + u$ and $ \subscr{f}{s} := \hh \circ \BOH\circ \BIH - \bigl(\SR I_m + \OR \diag{\BIH\circ\BIH} \bigr)w + \bar u$. We have:
\begin{align*}
\subscr{J}{nn} &= - \NR I_n + \BI \diag{w} \trasp{\BO} \diag{\Phi'(x)},\\
\subscr{J}{ns} &= \BI \diag{\trasp{\BO} \Phi(x)},\\
\subscr{J}{sn} &= \diag{\hh}\bigl(\diag{ \trasp{\BO}\Phi(x)}\trasp{\BI} + \diag{\trasp{\BI}\Phi(x)}\trasp{\BO}\bigr) \diag{\Phi'(x)} - 2\OR\diag{w}\diag{\BIH}\trasp{\BI} \diag{\Phi'(x)},\\
\subscr{J}{ss} &= - \SR I_m - \OR \diag{\BIH}\diag{\BIH}.
\end{align*}

\subsubsection{Firing-rate-Oja Model}\label{appx:jacobian_fro}
It is $\subscr{f}{n} := - \NR \fr + \Phi\bigl(\BI \diag{w} \trasp{\BO} \fr + u\bigr)$ and $\subscr{f}{s} := \hh \circ \BOH\circ \BIH - \bigl(\SR I_m + \OR \diag{\BIH\circ\BIH} \bigr)w + \bar u$. We have:
\begin{align*}
\subscr{J}{nn} &= - \NR I_n+ \BI \diag{w} \trasp{\BO} \left[\Phi'\left(\BI \diag{w} \trasp{\BO} \fr + u\right)\right],\\
\subscr{J}{ns} &= \BI \diag{\trasp{\BO}\fr} \left[\Phi'\bigl(\BI \diag{w} \trasp{\BO} \fr + u\bigr)\right],\\
\subscr{J}{sn} &= \diag{\hh}\bigl(\diag{ \trasp{\BO}\Phi}\trasp{\BI} + \diag{\trasp{\BI}\Phi}\trasp{\BO}\bigr) \diag{\Phi'}\-2\OR\diag{w}\diag{\BIH}\trasp{\BI} \diag{\Phi'},\\
\subscr{J}{ss} &= - \SR I_m - \OR \diag{\BIH}\diag{\BIH}.
\end{align*}

\subsection{Computation of the Spectral Abscissa}
We present here the detailed calculations for the computation of the spectral abscissa for the different models we have analyzed.
\subsubsection{Hopfield-Hebbian Model}\label{apx:spectral_abscissa_hopfield_hebbian}
Consider the matrix 
\[
\JacMTildeHH =
\begin{bmatrix}
- \NR +\domax (\hmax \phimax^2 + \ubarmax)/\SR & \domax \phimax \\
2\hmax\phimax & -\SR
\end{bmatrix}.
\]
We recall that the spectral abscissa of a matrix is the maximum among the real part of the elements in its spectrum.
Being $\JacMTildeHH$ a $2 \times 2$ matrix, its eigenvalues are the zero of the characteristic polynomial
\begin{align*}
p(&\lambda) = \lambda^2 - \tr\bigl(\JacMTildeHH \bigr) \lambda + \det\bigl(\JacMTildeHH\bigr),
\end{align*}
where $\tr(\JacMTildeHH ) = -\bigl(\SR + \NR - (\bmax+ \ubarmax\domax)/\SR\bigr)$ and $\det\bigl(\JacMTildeHH\bigr) = \NR\SR -3 \bmax -\ubarmax\domax$.
Under assumption~\eqref{eq:conditon_cn_cs} it is $\gh>0$.
With this notation we can write
\begin{align*}
p(\lambda) &= \lambda^2 + \Bigl(\frac{\constHH + \gh}{\SR}\Bigl)\lambda + \gh.
\end{align*}
We have $p(\lambda) = 0$ if and only if
\begin{align*}
\lambda_{1} &= -\frac{\constHH + \gh-\sqrt{\left(\constHH + \gh\right)^2 - 4\gh\SR^2}}{2\SR},\\
\lambda_{2} &=-\frac{\constHH + \gh+ \sqrt{\left(\constHH + \gh\right)^2 - 4\gh\SR^2}}{2\SR}.
\end{align*}
We observe that if $\Delta \leq 0$ then $\operatorname{Re}\bigl(\lambda_1\bigr) = \operatorname{Re}\bigl(\lambda_2\bigr) = \alpha\bigl(\JacMTildeHH\bigr)$, 
while if $\Delta > 0$, then $\alpha\bigl(\JacMTildeHH\bigr) = \lambda_1 > \lambda_2$.
Hence we now analyze the sign of $\Delta$. First we observe that
\begin{align*}
\Delta &= \frac{(\SR^2 + \gh +2\bmax)^2 - 4\gh\SR^2}{\SR^2}=\frac{(\gh - \SR^2)^2 +4\bmax(\bmax +\SR^2 + \gh)}{\SR^2}.
\end{align*}
Assuming condition~\eqref{eq:conditon_cn_cs} and being $\bmax$ and $\SR$ non negative, it always results $\Delta > 0$, so that $\alpha(\JacMTildeHH) := \contrateHH = \lambda_{1}.$
\subsubsection{Firing-rate-Hebbian Model}\label{apx:spectral_abscissa_fr}
Given the matrix 
\[
\JacMTildeFH=
\begin{bmatrix}
\domax (\hmax \psimax^2 + \ubarmax)/\SR - \NR & \domax \psimax/\NR \\
2\hmax\psimax & -\SR
\end{bmatrix},
\]
we have
\begin{align*}
p(\lambda) = \lambda^2 + \frac{1}{\SR}\left(\constFH + \gfr\right)\lambda + \gfr.
\end{align*}
We observe that under assumption~\eqref{eq:conditon_cn_cs_fr} it is $\gfr>0$.
We have $p(\lambda) = 0$ if and only if
\begin{align*}
\lambda_{1} &= -\frac{\constFH + \gfr - \sqrt{(\constFH + \gfr)^2 - 4\gfr\SR^2}}{2\SR},\\
\lambda_{2} &= -\frac{\constFH + \gfr + \sqrt{(\constFH + \gfr)^2 - 4\gfr\SR^2}}{2\SR}.
\end{align*}
We observe that
\begin{align*}
\Delta &= \frac{1}{\SR^2}\left(\left(\SR^2 + \gfr + \frac{2}{\NR}\amax\right)^2 - 4\gfr\SR^2\right)
=\frac{1}{\SR^2}\left(\left(\gfr - \SR^2\right)^2 +\frac{4}{\NR}\amax\left(\frac{\amax}{\NR} +\gfr + \SR^2\right)\right).
\end{align*}
Therefore, under condition~\eqref{eq:conditon_cn_cs_fr} and being $\amax, \NR, \SR$ non negative, it again always results $\Delta > 0$, so that $\alpha(\JacM) := \contrateFH = \lambda_{1}$. 

\subsubsection{Hopfield-Oja Model}\label{apx:spectral_abscissa_hopfield_oja}
Consider the matrix 
\[
\JacMTildeHO =
\begin{bmatrix}
\domax (\hmax \psimax^2 + \ubarmax)/\SR- \NR & \domax \phimax \\
2\phimax(\hmax+\OR(\hmax\phimax^2+\ubarmax)/\SR) & -\SR
\end{bmatrix}
\]
and its characteristic polynomial
\begin{align*}
p(\lambda)&= \lambda^2 + \left(\frac{\constHO + \go}{\SR}\right)\lambda + \go.
\end{align*}
We have $p(\lambda) = 0$ if and only if
\begin{align*}
\lambda_{1}&= -\frac{\constHO + \go-\sqrt{\left(\constHO + \go\right)^2 - 4\go\SR^2}}{2\SR},\\
\lambda_{2} &=-\frac{\constHO + \go + \sqrt{\left(\constHO + \go \right)^2 - 4\go\SR^2}}{2\SR}.
\end{align*}
It is
\begin{align*}
\Delta &= \frac{\left(\SR^2 + \go +2\bmax +2\phi^2\OR/\SR(\bmax+\domax\ubarmax)\right)^2}{\SR^2}- 4\go\\
&=\frac{(\go - \SR^2)^2 +4\bmax(\bmax +\SR^2 + \go)}{\SR^2}+\frac{4\phi^4\OR^2/\SR^2(\bmax+\domax\ubarmax)^2}{\SR^2}\\
&\quad +\frac{4\phi^2\OR/\SR(\bmax+\domax\ubarmax)\left(\SR^2 +\go + 2\bmax\right)}{\SR^2}.
\end{align*}
Therefore, under condition~\eqref{eq:conditon_cn_cs_oja} and being $\bmax$, $\OR$, and $\SR$ non negative, it always results $\Delta > 0$, so that $\alpha(\JacMTildeHO) := \contrateHO = \lambda_{1}$.
\subsubsection{Firing-rate-Oja Model}\label{apx:spectral_abscissa_firingrate_oja}
We consider the matrix
\[
\JacMTildeFO =
\begin{bmatrix}
\displaystyle\frac{\domax (\hmax \psimax^2 + \ubarmax)}{\SR} - \NR & \displaystyle \frac{\domax \psimax}{\NR}\\
\displaystyle 2\phimax\left(\hmax+\frac{\OR}{\SR}(\hmax\phimax^2+\ubarmax)\right) & -\SR
\end{bmatrix}
\]
and its characteristic polynomial
\begin{align*}
p(\lambda)&= \lambda^2 + \left(\frac{\constFO + \gof}{\SR}\right)\lambda + \gof,
\end{align*}
where for simplicity of notation we have defined $\amax := \domax \hmax \psimax^2$, $\constFO := \SR^2 + 2\amax/\NR +2\phimax^2\frac{\OR}{\SR\NR}(\amax+\domax\ubarmax)$ and $\gof = \NR\SR -\amax\left(1 + \frac{2}{\NR}\right) -\domax\ubarmax -2\frac{\OR}{\SR\NR}\phimax^2\left(\amax + \domax\ubarmax\right)$. We observe that under assumption~\eqref{eq:conditon_cn_cs_oja} it is $\gh>0$.

We have $p(\lambda) = 0$ if and only if
\begin{align*}
\lambda_{1}&= -\frac{\constFO + \gof-\sqrt{\left(\constFO + \gof\right)^2 - 4\gof\SR^2}}{2\SR},\\
\lambda_{2} &=-\frac{\constFO + \gof + \sqrt{\left(\constFO + \gof \right)^2 - 4\gof\SR^2}}{2\SR}.
\end{align*}
The discriminant of $p(\lambda)$ is
\begin{align*}
\Delta &= \displaystyle \frac{\left(\SR^2 + \gof +2\frac{\amax}{\NR} +2\phimax^2\frac{\OR}{\SR\NR}(\amax+\domax\ubarmax)\right)^2}{\SR^2} - 4\gof\\
&=\displaystyle \frac{(\gof - \SR^2)^2 +4\frac{\amax}{\NR}(\frac{\amax}{\NR} +\SR^2 + \gof)}{\SR^2} \displaystyle +\frac{4\phimax^4\left(\frac{\OR}{\SR\NR}\right)^2(\amax+\domax\ubarmax)^2}{\SR^2}\\
&\displaystyle \quad +\frac{4\phimax^2\frac{\OR}{\SR\NR}(\amax+\domax\ubarmax)\left(\SR^2 +\gof + 2\frac{\amax}{\NR}\right)}{\SR^2}.
\end{align*}
Therefore, under condition~\eqref{eq:conditon_cn_cs_oja} and being $\bmax$, $\OR$, and $\SR$ non negative, it always results $\Delta > 0$, so that $\alpha(\JacMTildeFO) := \contrateFO = \lambda_{1}$.
\end{extendedArxiv}

\bibliographystyle{plainurl+isbn}
\bibliography{main}
\end{document}